\documentclass[12pt,a4paper,oneside]{amsart}

\usepackage[left=1.1in, right=0.8in, top=1.0in, bottom=1.0in]{geometry}
\usepackage{xcolor}
\usepackage{comment}
\usepackage{graphicx}
\usepackage{setspace}
\usepackage{indentfirst} 
\usepackage{cases}
\usepackage{wrapfig}
\usepackage[toc,page]{appendix}
\usepackage{amsmath}
\usepackage{amsthm}
\usepackage{amssymb}
\usepackage{enumerate}
\usepackage{mathrsfs}
\usepackage{dcolumn}
\usepackage{hyperref}
\usepackage{cite}
\usepackage{textgreek}

\newcolumntype{L}{D{.}{.}{2,5}}

\newtheorem{theorem}{Theorem}

\newtheorem{proposition}{Proposition}
\newtheorem{lemma}{Lemma}
\newtheorem{corollary}[proposition]{Corollary}
\newtheorem{mydef}{Definition}

\newtheorem{question}{Question}
\newtheorem{remark}{Remark}
\newtheorem{example}{Example}
\theoremstyle{definition}

\DeclareMathOperator{\cl}{cl}

\DeclareMathOperator{\co}{co}
\DeclareMathOperator{\wcl}{wcl}

\DeclareMathOperator{\lev}{lev}
\DeclareMathOperator{\CAT}{CAT}

\begin{document}

\title[On a weak topology for Hadamard spaces]{On a weak topology for Hadamard spaces\\ and its applications} 
\author[1]{Arian B\"erd\"ellima}

\address{
 Institut f\"ur Mathematik}
\address{ Technische Universit\"at Berlin}
\address{Stra\ss e des 17. Juni 136}
\address{ D-10623 Berlin, Germany}
 \thanks{Electronic address: \texttt{berdellima@math.tu-berlin.de/berdellima@gmail.com},}
\maketitle

\begin{abstract}
	We investigate if an existing notion of weak sequential convergence in a Hadamard space can be induced by a topology. We provide an answer on what we call weakly proper Hadamard spaces. A notion of dual space is proposed and it is shown that our weak topology and dual space coincide with the standard ones in the case of a Hilbert space. Moreover we introduce the space of geodesic segments and a corresponding weak topology, and we show that this space is homeomorphic to its underlying Hadamard space. As an application of it we show the existence of a geodesic segment that acts as direction of steepest descent for a geodesically differentiable function whose geodesic derivative satisfies certain properties. Finally we extend several results from classical functional analysis to the setting of Hadamard spaces, and we compare our topology with other existing notions of weak topologies. 
\end{abstract} 


\section{Introduction}
Every metric space can be equipped with a canonical topology induced by its metric. This topology 
characterizes the usual (strong) convergence of sequences. It has been of interest to analysts to study weaker notions of convergence for sequences in a metric space and if possible to identify topologies that induce such convergences.
In \cite[Lim, 1976]{Lim} introduced the notion of $\Delta$-convergence for a general metric space. This form of convergence is weaker then metric convergence, therefore offering a window into what it might mean, for the standard concept of weak convergence in a linear space, in the setting of a possibly non-linear metric space.
In \cite{Kirk} the notion of $\Delta$-convergence was adopted in the setting of a $\CAT(0)$ space. These are uniquely geodesic metric spaces where every so-called geodesic triangle is thinner than its comparison triangle in the Euclidean plane. A $\CAT(0)$ space is often referred to as a space of non-positive curvature in the sense of Alexandrov \cite{Alexandrov}. For an extensive treatment of these spaces, 
see, e.g., \cite{Brid, Burago, Alexander2019}.
A complete $\CAT(0)$ space is known as a Hadamard space and we denote it by $(H,d)$. One basic example of a Hadamard space is Hilbert spaces and the notion of $\Delta$-convergence coincides with the usual weak convergence there.

An open question has been the identification of a topology $\tau_{\Delta}$ on $H$ such that for a given bounded sequence $(x_n)\subset H$ and $x\in H$ we have $x_n\overset{\Delta}\to x$ if and only if $(x_n)$ converges sequentially to $x$ with respect to $\tau_{\Delta}$. It is known that a bounded sequence $x_n\overset{\Delta}\to x$ if and only if the metric projections $P_{\gamma}x_n$ converge to $x$ along any geodesic $\gamma$ starting at $x$. It is in the light of this equivalent definition, which we wish to refer to as the \emph{weak convergence} of a bounded sequence, that \cite[Lytchak--Petrunin]{Lytchak} recently settled the open problem. However the question of constructing a topology that also characterizes weak convergence of unbounded sequences remained open. Dropping the requirement for boundedness in our study of weak convergence highlights one of the key differences between a general Hadamard space and Hilbert spaces, since for the latter the Banach--Steinhaus theorem implies that every weakly convergent sequence must indeed be bounded. 
Inspired by a question of Ba\v cak~\cite{Bacak3}, the author had offered in his thesis \cite{BerdellimaPhD} such a weak topology $\tau_w$ on what we call  \emph{weakly proper} Hadamard spaces. Weakly properness entails a certain topological regularity condition for the open sets of $\tau_w$ (see Definition \ref{d:weakly-proper}).
Although initially we did not know whether or not all Hadamard spaces were weakly proper, \cite[\S 4]{Lytchak} provides a counterexample. Nevertheless a wide range of Hadamard spaces are in fact weakly proper among which Hilbert spaces and any locally compact Hadamard space.
 Our work is a presentation of the main results in \cite{BerdellimaPhD} and develops as follows. In \S\ref{s:preliminaries} we provide some preliminary definitions and results. Next in \S\ref{sec:weak}, inspired by a suggestion of Ba\v cak\cite{Bacak3}, we construct the weak topology $\tau_w$ and introduce the notion of a weakly proper Hadamard space. We show that on a weakly proper Hadamard space, weak sequential convergence and weak convergence coincide (Theorem \ref{wtopBacak}) and in particular a weakly proper Hadamard space is always Hausdorff in $\tau_w$ (Lemma \ref{wtopHaus}). Furthermore we show that many useful properties of the weak topology in a Hilbert space, extend naturally to the setting of a general Hadamard space. For example a (geodesically) convex set is closed if and only if it is weakly closed (Theorem \ref{wtopcvx}) and the analogue of Mazur's Lemma also holds true (Theorem \ref{wtopMazur}). Additionally it is shown that in a separable Hadamard space weak sequential compactness always implies weak compactness (Theorem \ref{lemma:seqcomp}), and provided that the space is also Hausdorff in $\tau_w$ then weak sequential compactness and weak compactness coincide on bounded sets (Corollary \ref{c:EberleinSmulyan}). A conditional result is given for unbounded sets when the underlying Hadamard space is separable and satisfies a certain regularity condition for its geodesics (Theorem \ref{thm:wcompact-wseqcompact}). Moreover in \S\ref{sec:weak}, differently from \cite{BerdellimaPhD} we give an (unconditional) proof that in a locally compact Hadamard space the metric topology and weak topology coincide (Theorem \ref{wtopstop}). Later in \S\ref{s:dual} we introduce the notion of a dual space and of a weak-* convergence together with its corresponding weak-* topology. By construction the dual space is itself a metric space (Lemma \ref{l:dual}) and that weak-* topology does indeed characterize weak-* convergence (Theorem \ref{wtop*}). In the case of a Hilbert space it is shown that both the weak topology and the dual space coincide with the standard notions (Proposition \ref{p:Hilbertweakcvg}). Moreover the Cartesian product of a Hilbert space and a locally compact Hadamard space is weakly proper (Corollary \ref{c:productweaklyprop}), therefore offering non-Hilbertian examples of weakly proper Hadamard spaces. In \S\ref{s:geo-seg}, with an eye towards convex analysis, we introduce the space of geodesic segments. Equipped with a supremum metric it is shown that the space of geodesic segments is isometric to the underlying Hadamard space, and therefore it is itself a Hadamard space (Theorem \ref{globaliso}). Moreover we can identify a corresponding notion of weak convergence and weak topology and prove that the space of geodesic segments and the underlying Hadamard space are topologically indistinguishable (Proposition \ref{homeomorphic}, Theorem \ref{weakgamma}).  As an application of this space we show the existence of a geodesic segment that acts as direction of steepest descent for a geodesically differentiable function whose geodesic derivative satisfies certain properties (Theorem \ref{descentloccompact}, \ref{descentgeneral}).
Finally in \S \ref{s:other-topo}, we compare our notion to existing notions of weak topologies in Hadamard spaces. In particular, we compare our topology to that of Monod \cite{Monod} and that of Kakavandi \cite{Kakavandi}. We close in \S\ref{ss:discussion} with a short discussion about the weak topology introduced lately by \cite[Lytchak--Petrunin]{Lytchak}.


\section{Preliminaries}
 \label{s:preliminaries}
 \subsection{Hadamard spaces and related notions} Let $(H,d)$ be a Hadamard space. A constant speed geodesic $\gamma:[0,1]\to H$  is a curve that satisfies $d(\gamma(s), \gamma(t))=|s-t| d(\gamma(0), \gamma(1))$ for all $s,t\in [0,1]$. For a given $x\in H$ we define $\Gamma_{x}(H)$ to be the set of all constant speed  geodesics $\gamma:[0,1]\to H$ such that $\gamma(0)=x$. Given $x,y\in H$ we denote by $[x,y]$ the geodesic segment joining $x$ and $y$ and by $x_t=(1-t)x\oplus t y$ the element in $[x,y]$ such that $d(x_t,x)=td(x,y)$ for $t\in[0,1]$.
 A geodesic triangle $\Delta(pqr)$ determined by three points $p,q,r\in H$ consists of three geodesic segments $[p,q],[q,r], [r,p]$. To each geodesic triangle $\Delta(pqr)$ there corresponds a comparison triangle $\Delta(\overline{p}\overline{q}\overline{r})$ in the Euclidean plane $\mathbb{E}^2$ with vertices $\overline{p},\overline{q},\overline{r}$ such that $d(p,q)=\|\overline{p}-\overline{q}\|, d(q,r)=\|\overline{q}-\overline{r}\|$ and $d(p,r)=\|\overline{p}-\overline{r}\|$. A point $\overline{x}\in[\overline p,\overline q]$ is a comparison point in $\mathbb E^2$ for the point $x\in[p,q]$ if $d(p,x)=\|\overline p-\overline x\|$. 
 By definition of a Hadamard space then for every $x\in[p,q]$ and $y\in[p,r]$ we have the so-called $\CAT(0)$ inequality 
 \begin{equation}
 \label{eq:CAT(0)-ineq}
 d(x,y)\leqslant\|\overline{x}-\overline{y}\|.
 \end{equation}
 
 The interior angle of $\Delta(\overline{p}\overline{q}\overline{r})$ at $\overline{p}$ is called the comparison angle between $q$ and $r$ at $p$ and it is denoted by $\overline{\angle}_p(q,r)$. Now let $\gamma$ and $\eta$ be two geodesic segments emanating from the same point $p$, i.e., $\gamma(0)=\eta(0)=p$. The Alexandrov angle between $\gamma$ and $\eta$ at $p$ is the number $\angle_p(\gamma,\eta)\in[0,\pi]$ defined as 
 \begin{equation}
 \label{eq:Alexandrovangle}
 \angle_p(\gamma,\eta):=\limsup_{t,t'\to 0}\overline{\angle}_p(\gamma(t),\eta(t'))
 \end{equation}
 In a Riemannian manifold, the Alexandrov angle coincides with the usual angle. 
 
 A set $C\subseteq H$ is (geodesically) convex if $x,y\in C$ implies that $[x,y]\subseteq C$. We denote by $\co C$ the convex hull of $C$, that is the smallest convex set in $H$ that entirely contains $C$. Let $d(x,C)=\inf\{y\in C\,:\,d(x,y)\}$ and denote by $P_Cx:=\{y\in H\,:\, d(x,y)=d(x,C)\}$ the metric projection of $x$ onto $C$. In general $P_Cx$ can be multivalued or even empty, however when $C$ is a closed and convex set then $P_Cx$ is nonempty and unique for every $x\in H$. Moreover by virtue of \cite[Theorem 2.1.12 ]{Bacak} the following inequality is satisfied
 \begin{equation}
 \label{eq:projection}
 d(y,P_Cx)^2+d(x,P_Cx)^2\leqslant d(x,y)^2,\quad\text{for all}\;y\in C.
 \end{equation}
From inequality \eqref{eq:CAT(0)-ineq} another equivalent characterization for $\CAT(0)$ spaces, and hence also for Hadamard spaces, follows
\begin{equation}
\label{eq:quadratic}
d(x_t,z)^2\leqslant(1-t)d(x,z)^2+td(y,z)^2-t(1-t)d(x,y)^2,\quad\text{for all}\;x,y,z\in H,\,t\in[0,1]. 
\end{equation} 
 
\subsection{$\Delta$-convergence}
Let $(X,d)$ be metric space. A sequence $(x_n)$ in $X$ $\Delta$-converge to $x$, written as $x_n\overset{\Delta}\to x$, if 
\begin{equation*}
 \limsup_{k\to+\infty}d(x_{n_k},x)\leqslant\limsup_{k\to+\infty}d(x_{n_k},y)
\end{equation*}
for every subsequence $(x_{n_k})$ of $(x_n)$ and for every $y\in X$. In Hadamard spaces $\Delta$-convergence has a natural interpretation in terms of the so called asymptotic centers. 
Let $(x_n)\subset H$ be a bounded sequence. 
For $y\in H$ define 
\begin{equation*}
r(y, (x_n)):=\limsup_{n\to \infty}d(x_n,y)
\quad \text{and let}\quad 
r((x_n)):=\inf_{y\in H}r((x_n),y) 
\end{equation*}
denote the asymptotic radius of $(x_{n})$. Then the set 
$
\{y\in H \, | \, r(y, (x_n)) = r((x_n))\}
$
consists of a single point, called the asymptotic center of $(x_{n})$. This is a direct consequence of the square of the metric being a strongly convex function, see inequality \eqref{eq:quadratic}. Accordingly, a bounded sequence $(x_n)\subset H$ is said to be $\Delta$-convergent to a point $x\in H$ whenever $x$ is the asymptotic center of the sequence $(x_{n})$. 
The notion of $\Delta$-convergence in $\CAT(0)$ spaces shares many properties with the usual notion of weak convergence in Banach spaces. For example, $\Delta$-convergence in $\CAT(0)$  satisfies the Kadec--Klee property (see~\cite[Kirk and Panyanak]{Kirk}) and a Banach--Saks-type property (see \cite[Ba\v cak]{Bacak}).

\subsection{Weak convergence}  We develop our work on a notion of weak convergence proposed originally by Jost \cite{Jost} (See also \cite{Sosov, Espi}.) We say that a sequence $(x_n)\subseteq H$  weakly converges to $x\in H$ if and only if $P_{\gamma}x_n\to x$ as $n\to+\infty$ along any geodesic $\gamma:[0,1]\to H$ starting at $x$, where  $P_{\gamma}$ denotes the projection to $\gamma$.  Analogue definition can be stated in terms of nets. We denote weak convergence by  $x_n\overset{w}\to x$. Notice that we do not require boundedness in this definition, but if
$(x_{n})$ is bounded, then $x_n\overset{\Delta}\to x$ if and only if $x_n\overset{w}\to x$, see, e.g., ~\cite[Proposition 3.1.3]{Bacak}. Therefore, our notion of weak convergence agrees with the notion of $\Delta$-convergence for bounded sequences in Hadamard spaces. For Hilbert spaces, our notion  of weak convergence agrees with the usual one, see discussion in \S \ref{ss:Hilbert}. 

\begin{remark}
	\label{r:uniqueness}
	Weak limits are unique. Indeed if $x_n\overset{w}\to x$ and $x_n\overset{w}\to y$ then the elementary inequality $d(x,y)\leqslant d(x,P_{\gamma}x_n)+d(P_{\gamma}x_n,y)=d(x,P_{\gamma}x_n)+d(P_{\tilde{\gamma}}x_n,y)\to 0$ implies $x=y$. Here $\gamma, \tilde \gamma:[0,1]\to H$ are geodesics such that $\gamma(0)=x$, $\gamma(1)=y$ and $\tilde \gamma(0)=y$, $\tilde \gamma(1)=x$. 
\end{remark}


\section{Identification of a weak topology}\label{sec:weak}

\subsection{Construction of open sets} 
Consider a Hadamard space  $(H,d)$. Following the suggestion by Ba\v cak \cite{Bacak3} we say a set $U\subset H$ is weakly open if for every $x\in U$ there exists some $\varepsilon>0$ and a finite family of geodesics $\gamma_1,\gamma_2,...,\gamma_n \in \Gamma_{x}(H)$ such that the set 
\begin{equation}
\label{eq:openset}
U_{x}(\varepsilon;\gamma_1,...,\gamma_n):=\{y\in H: d(x,P_{\gamma_i}y)<\varepsilon\hspace{0.2cm}\forall i=1,2,...,n\}
\end{equation}
is contained in $U$. Here $P_{\gamma_i}y$ denotes the projection of $y$ to the geodesic $\gamma_i$. For a given geodesic segment $\gamma\in\Gamma_x(H)$ we call $U_x(\varepsilon;\gamma)$ an elementary set.
\begin{proposition}
	\label{p:weak topology}
	The collection of weakly open sets $U$ together with the empty set $\emptyset$ define a topology $\tau_w$ on $H$ and we call it the weak topology on $H$. 
\end{proposition}
\begin{proof}
	It is clear that $H\in \tau_w$ since if $x\in H$ then for any finite family of geodesic segments $\gamma_1,...,\gamma_n\in\Gamma_x(H)$ the set $U_x(\varepsilon;\gamma_1,...,\gamma_n)\subseteq H$ for every $\varepsilon>0$. The empty set $\emptyset$ is in $\tau_w$ by definition. 
	
	Moreover for any collection $\{U_i\}_{i\in I}$ where $I$ is some index set and $U_i$ is weakly open for all $i\in I$ its union is weakly open. To see this let $x\in \bigcup_{i\in I}U_i$ then $x\in U_j$ for some $j\in I$. Since $U_j$ is weakly open then there exist $\varepsilon>0$ and $\gamma_1,...,\gamma_n\in\Gamma_x(H)$ such that $U_x(\varepsilon;\gamma_1,...,\gamma_n)\subseteq U_j\subseteq \bigcup_{i\in I}U_i$. Hence $\bigcup_{i\in I}U_i$ is weakly open.
	
	Now let $U_1$ and $U_2$ be two weakly open sets and let $x\in U_1\cap U_2$. Then there exist $\varepsilon_1,\varepsilon_2>0$ and geodesic segments $\gamma_1,...,\gamma_n,\eta_1,...,\eta_m\in\Gamma_x(H)$ such that $U_{x}(\varepsilon_1;\gamma_1,...,\gamma_n)\subseteq U_1$ and $U_{x}(\varepsilon_2;\eta_1,...,\eta_m)\subseteq U_2$. Let $\varepsilon:=\min\{\varepsilon_1,\varepsilon_2\}$ and consider the set $U_x(\varepsilon; \gamma_1,...,\gamma_n,\eta_1,...,\eta_m)$. By construction it follows that 
	$$U_x(\varepsilon; \gamma_1,...,\gamma_n,\eta_1,...,\eta_m)\subseteq U_{x}(\varepsilon_1;\gamma_1,...,\gamma_n)\cap U_{x}(\varepsilon_2;\eta_1,...,\eta_m)\subseteq U_1\cap U_2$$
	Therefore $U_1\cap U_2$ is weakly open.
\end{proof}
Following standard terminology in topology, see e.g. \cite{Bert, Kelley, Munkres} we study the properties of $\tau_w$.
First to justify the name weakly open we need to show that any set $U$ is open in the usual metric topology.
\begin{lemma}
	\label{l:lemma1}
	Let $(H,d)$ be a Hadmard space. Then $U_{x}(\varepsilon;\gamma)$ is open in the metric topology for every $x\in H$ and $\gamma\in\Gamma_x(H)$.
\end{lemma}

\begin{proof}
	From inequality \eqref{eq:projection} projections $P_C$ onto closed convex sets $C$ are nonexpansive operators, i.e. $d(P_{C}x,P_Cy)\leqslant d(x,y)$ for all $x,y\in H$. In particular, $P_{\gamma}$ is nonexpansive since every geodesic segment $\gamma$ is a closed convex set, i.e.,  
	\begin{equation}
	\label{eq:projne}
	d(P_{\gamma}x,P_{\gamma}y)\leqslant d(x,y)\hspace{0.2cm}\forall x,y\in H .
	\end{equation}
	Now let $y\in U_{x}(\varepsilon;\gamma)$. Then $s:=d(x,P_{\gamma}y)$ satisfies $s<\varepsilon$.
	Therefore, the open geodesic ball $B(y,\varepsilon-s):=\{z\in H: d(x,y)<\varepsilon-s\}$ is contained in $U_{x}(\varepsilon;\gamma)$. Indeed, let $z\in B(y,\varepsilon-s)$. Then 
	$$d(x,P_{\gamma}z)\leqslant d(x,P_{\gamma}y)+d(P_{\gamma}y, P_{\gamma}z)\leqslant d(x,P_{\gamma}y)+d(y,z) < s + (\varepsilon-s) =\varepsilon.$$
 Therefore, 
	$z \in U_{x}(\varepsilon;\gamma)$ and the proof is completed. 
\end{proof}
We introduce the following notion which plays a central role in this work:
\begin{mydef}
	\label{d:weakly-proper}
	We call a Hadamard space weakly proper if every $x\in H$ is an interior point for every elementary set $U_x(\varepsilon;\gamma)$ in the weak topology $\tau_w$, i.e., if for every $x\in H$ and every  $\gamma \in \Gamma_{x}(H)$ there exists $V\in \tau_w$ such that $x \in V\subseteq U_x(\varepsilon;\gamma)$. 
\end{mydef}
Notice that weak properness does not imply that elementary sets are weakly open, however $H$ is weakly proper whenever all elementary sets are weakly open. Moreover the implication $x_n\to x$ yields $x_n\overset {w}\to x$ is evident, for if $\lim_{n\to \infty}d(x,x_n)=0$ then $d(x,P_{\gamma}x_n)=d(P_{\gamma}x,P_{\gamma}x_n)\leqslant d(x,x_n)$ implies $\lim_{n\to \infty}d(x,P_{\gamma}x_n)=0$ for all $\gamma\in\Gamma_x(H)$.


 
We say that a sequence $(x_n)\subseteq H$ sequentially converges in $\tau_w$ to an element $x\in H$, and we denote this by $x_n\overset{\tau_w}\to x$, if and only if for every weakly open set $U$ containing $x$ all but finitely many elements of the sequence $(x_n)_{n\in\mathbb{N}}$ are in $U$. 

\begin{theorem}
	\label{wtopBacak}
	Let $(x_n)_{n\in\mathbb{N}}\subset H$, and let $x\in H$. If $x_n\overset{w}\to x$ then $x_n\overset{\tau_w}\to x$. Moreover, if the Hadamard space is weakly proper then $x_n\overset{\tau_w}\to x$ implies $x_n\overset{w}\to x$.
\end{theorem}
\begin{proof} 
	Let $x_n\overset{w}\to x$. Then for every $\gamma\in \Gamma_x(H)$ we have that $\lim_{n\to \infty}d(x,P_{\gamma}x_n)=0$, or equivalently  
	$x_n\in U_x(\varepsilon;\gamma)$ for all sufficiently large $n$. Let $U\in \tau_w$ contain $x$. Then there exist $\gamma_1,...,\gamma_n\in\Gamma_x(H)$ and $\varepsilon>0$ such that $U_x(\varepsilon;\gamma_1,...,\gamma_n)\subseteq U$. Since $x_n\in U_x(\varepsilon;\gamma_1,...,\gamma_n)$ for all sufficiently large $n$, we obtain that  $x_n\in U$ for all sufficiently large $n$.
	
	Let $H$ be weakly proper.
	Suppose that 
	$x_n\overset{\tau_w}\to x$ but $x_n\overset{w}\nrightarrow x$. Then there exists $\gamma\in\Gamma_x(H)$ and $\varepsilon>0$ such that $x_n\notin U_x(\varepsilon;\gamma)$ for infinitely many $n$. Weakly properness implies that there is an open set $V\in\tau_w$ containing $x$ such that $V\subseteq U_x(\varepsilon;\gamma)$. Therefore, $x_n\notin V$ for infinitely many $n$. However, this contradicts $x_n\overset{\tau_w}\to x$.
	\end{proof}
It is clear from Theorem \ref{wtopBacak} that $x_n\to x$ implies $x_n\overset{\tau_w}\to x$.

\begin{lemma}
\label{wtopHaus}
 Every weakly proper Hadamard space is Hausdorff with respect to the weak topology $\tau_w$.
\end{lemma}
\begin{proof}
 Let $x,y\in H$ be two distinct points and $\gamma, \tilde \gamma:[0,1]\to H$ the geodesics connecting $x$ with $y$ such that $\gamma(0)=x$, $\gamma(1)=y$ and $\tilde \gamma(0)=y$, $\tilde \gamma(1)=x$. Let $l_{\gamma}:=d(x,y)$, and let $\varepsilon\in (0,l_{\gamma})$. Note that $l_{\gamma}=l_{\tilde\gamma}$ and $P_{\gamma}z=P_{\tilde\gamma}z$ for all $z\in H$.
Define the sets 
$$U_x(\varepsilon; \gamma):=\{z\in H: d(x,P_{\gamma}z)<\varepsilon\}$$
and 
$$U_y(l_{\gamma}-\varepsilon; \tilde\gamma):=\{z\in H: d(y,P_{\tilde\gamma}z)<l_{\gamma}-\varepsilon\} .$$ 
Suppose there is some $z_0\in U_x(\varepsilon; \gamma)\cap U_y(l_{\gamma}-\varepsilon; \tilde\gamma)$ then $d(x,P_{\gamma}z_0)<\varepsilon$ and $d(y,P_{\tilde\gamma}z_0)<l_{\gamma}-\varepsilon$ would imply $$l_{\gamma}=d(x,y)\leqslant d(x,P_{\gamma}z_0)+d(y,P_{\gamma}z_0)= d(x,P_{\gamma}z_0)+d(y,P_{\tilde\gamma}z_0)<l_{\gamma}$$
which is impossible. Therefore $U_x(\varepsilon; \gamma)\cap U_y(l_{\gamma}-\varepsilon; \tilde\gamma)=\emptyset$. Since $(H,d)$ is weakly proper there are $V,W\in\tau_w$ such that $x\in V\subseteq U_x(\varepsilon; \gamma)$ and $y\in W\subseteq U_y(l_{\gamma}-\varepsilon; \tilde\gamma)$. Then $V\cap W\subseteq U_x(\varepsilon; \gamma)\cap U_y(l_{\gamma}-\varepsilon; \tilde\gamma)=\emptyset$. 
\end{proof}

\subsection{Convex sets and compactness}
We say that a set $S\subseteq H$ is weakly closed if it is closed with respect to $\tau_w$. 

\begin{theorem}
\label{wtopcvx}
Let $(H,d)$ be a Hadamard space.
A convex set $C\subseteq H$ is strongly closed if and only if it is weakly closed.
\end{theorem}
\begin{proof}
Any weakly closed set is strongly closed. So let $C$ be a strongly closed convex set. We show that $C$ is weakly closed.
It suffices to show that $H\setminus C$ is weakly open.
Let $y\in H\setminus C$. Then $P_Cy$ exists and is unique. Let $\gamma:[0,1]\to H$ be the geodesic connecting $y$ with $P_Cy$ such that $\gamma(0)=y$ and $\gamma(1)=P_Cy$. For $\varepsilon\in(0,l(\gamma))$, where $l(\gamma):=d(y,P_Cy)$ is the length of the geodesic $\gamma$, consider the elementary set 
$U_y(\varepsilon;\gamma)$.
We need only show that $U_y(\varepsilon;\gamma)\cap C=\emptyset$. Let $x\in C$, and let $z\in H$ be arbitrary.
Since both $C$ and $\gamma$  are strongly closed and convex, an application of the inequality \eqref{eq:projection} yields:
$$d(x,z)^2\geqslant d(x,P_{C}z)^2+d(P_Cz,z)^2$$
and 
$$d(x,P_Cy)^2\geqslant d(x,P_{\gamma}x)^2+d(P_{\gamma}x,P_Cy)^2 ,$$
where we have used that $x$ lies in the convex set $C$ for the first inequality and that $P_Cy$ lies in the convex set $ \gamma$ for the second inequality. 
Now let $z=P_\gamma x$ be the projection of $x$ to $\gamma$. Since $z \in \gamma$, we have that $P_C z= P_C y$, see \cite{Brid}. Then the above two inequalities imply that $P_Cy = z= P_{\gamma}x$. From $d(y,P_{\gamma}x)=d(y,P_Cy)>\varepsilon$ for all $x\in C$ it follows that $U_y(\varepsilon;\gamma)\cap C=\emptyset$. Since $y\in H\setminus C$ is arbitrary then $U_y(\varepsilon;\gamma)\subseteq H\setminus C$ yields the claim.
\end{proof}


\begin{theorem}
\label{wtopMazur}(Mazur's Lemma) 
Let $(x_n)\subseteq H$ be a sequence such that $x_n\overset{w}\to x$ for some $x\in H$.  Then there exists a function $N:\mathbb{N}\to\mathbb{N}$ and a sequence $(y_n)\subseteq H$ such that  $y_n\to x$ and $y_n\in \co(\{x_1,x_{2},...,x_{N(n)}\})$ for all $n\in\mathbb{N}$. 
\end{theorem}
\begin{proof}
Weak convergence $x_n\overset{w}\to x$ implies by virtue of Theorem	\ref{wtopBacak} that $x\in \wcl\{x_1,x_2,...\}$, where $\wcl\{x_1,x_2,...\}$ is the weak closure of $\{x_1,x_2,...\}$. Moreover $$\{x_1, x_2,...\}\subseteq\co(\{x_1, x_2,...\})$$ yields $\wcl\{x_1,x_2,...\}\subseteq\wcl\co(\{x_1,x_2,...\})$, hence we have that $x\in \wcl\co(\{x_1,x_2,...\})$. The strong closure $\cl\co(\{x_1,x_2,...\})$ of the convex set $\co(\{x_1,x_2,...\})$ is a closed convex set. Indeed, if $C$ is (geodesically) convex, then so is $\cl C$. In order to see this, let $u,v \in \cl C$. Consider sequences $(u_n)$ and $(v_n)$ in $C$ such that $u_n \to u$ and $v_n\to v$. Let $\gamma$ be the geodesic connecting $u$ with $v$, and let $\gamma_n$ be the geodesics connecting $u_n$ with $v_n$. From the characterization inequality \eqref{eq:quadratic} we obtain that
\begin{align*}
d(\gamma(t),\gamma_n(t))^2 
\leqslant &(1-t) d(u,\gamma_n(t))^2 + t d(v,\gamma_n(t))^2 - t(1-t)d(u,v)^2  \\
\leqslant &(1-t) \left( (1-t) d(u,u_n)^2 + t d(u,v_n)^2 - t(1-t)d(u_n,v_n)^2 \right) \\
&+ t \left( (1-t) d(v,u_n)^2 + t d(v,v_n)^2 - t(1-t)d(u_n,v_n)^2 \right) \\
&- t(1-t)d(u,v)^2 .
\end{align*}
Taking the limit $n\to \infty$ yields $d(\gamma(t),\gamma_n(t)) \to 0$. Hence $\cl C$ and therefore $\cl\co(\{x_1,x_2,...\})$ are indeed convex. By Theorem \ref{wtopcvx} $\cl\co(\{x_1,x_2,...\})$ is weakly closed.  It follows that $$x\in \wcl\co(\{x_1,x_2,...\})\subseteq\cl\co (\{x_1,x_2,...\} ).$$ 
Then there exists some sequence $(y_n)\subseteq \co (\{x_1,x_2,...\})$ such that $y_n\to x$. Additionally, we have that $\co (\{x_1,x_2,...\})=\bigcup_{k\in\mathbb{N}}\co(\{x_1,x_{2},...,x_{k}\})$. Hence 
$y_n\in \co(\{x_1,x_{2},...,x_{k(n)}\})$ for some $k(n)$. For each $n$ set $N(n):=k(n)$.
\end{proof}

A set $K\subseteq H$ is called weakly sequentially compact if every sequence in $K$ has a weakly convergent subsequence. Similarly, $K$ is called $\tau_w$-sequentially compact if every sequence in $K$ has a $\tau_w$-convergent subsequence. Weak sequential compactness implies $\tau_w$-sequential compactness, and these notions coincide on a weakly proper Hadamard space.
Finally, a set $K$ is called weakly compact if for any open cover in $\tau_w$ of $K$ there is a finite subcover of $K$. 

\begin{lemma}\cite[Proposition 3.1.2]{Bacak}
\label{l:Bacak}
Every bounded sequence in a Hadamard space has a weakly convergent subsequence.
\end{lemma}

\begin{lemma}
	[\hspace{-0.03em}{\cite[Lemma 3.2.1]{Bacak}}]
	\label{l:Bacak2}
	Let $K\subseteq H$ be a closed convex set and $(x_n)_{n\in\mathbb{N}}\subset K$. If $x_n\overset{w}\to x$ then $x\in K$.
\end{lemma}

Lemmas \ref{l:Bacak} and \ref{l:Bacak2} immediately imply the following result:

 \begin{theorem}
 \label{wseqcompact}
 Any bounded closed convex set $K$ in a Hadamard space is weakly sequentially compact and therefore $\tau_w$-sequentially compact.
 \end{theorem}

\begin{theorem}
\label{lemma:seqcomp}
Let $(H,d)$ be separable. Then any weakly sequentially compact set $K \subset H$ is weakly compact. 
\end{theorem}
\begin{proof}
The proof proceeds by contradiction. Suppose that $K\subseteq H$ is weakly sequentially compact but not weakly compact. Then there exists some open cover $\{U_i\}_{i\in I}$ of $K$ in $\tau_w$  that has no finite subcover. By assumption $(H,d)$ is separable. Hence $(H,d)$ is a Lindel\"of space, i.e., every open cover (in the strong topology) has a countable subcover; see, e.g., \cite[Theorem 6.7]{Bert}. Since
$\{U_i\}_{i\in I}$ is an open cover in the usual metric topology, there exists a countable subcover $\{U_j\}_{j\in J}$. 
Let  $V_n:=\bigcup_{j=1}^nU_{j}$. 
Then $W_n:=H\setminus V_n$ is weakly closed for all $n$. Moreover, the family of sets $W_n$ satisfies $W_{n+1}\subseteq W_n$.  
Because $V_n$ cannot cover $K$, we have that $W_n\cap K$ is nonempty for every $n\in\mathbb{N}$. Let $x_n\in W_n\cap K$. Since $K$ is weakly sequentially compact, and consequently $\tau_w$-sequentially compact, the sequence $(x_n)$ has a subsequence $(x_{n_k})$ that converges in $\tau_w$ to some element $x^*\in K$.
Let $\mathcal{U}_w(x^*)$ denote the collection of weakly open sets containing $x^*$.
Then for each $U\in \mathcal{U}_w(x^*)$ and for each $n\in \mathbb{N}$ there exists $m\geqslant n$ such that $U\cap W_m\neq\emptyset$, and in particular $U\cap W_n\neq\emptyset$, implying that $x^*\in\wcl W_n=W_n$. Since $n$ is arbitrary, we have that $x^*\in\bigcap_nW_n$. Hence $x^*\in \bigcap_nW_n\cap K$.
Therefore, $\bigcap_{n\in\mathbb{N}}W_n\cap K\neq\emptyset$, which together with $\bigcap_{n\in\mathbb{N}}W_n\cap K\subsetneq K$,
yields that $K\supsetneq K\setminus( \bigcap_{n\in\mathbb{N}}W_n\cap K)=K\setminus\bigcap_{n\in\mathbb{N}}(K\setminus V_n)=\bigcup_{n\in\mathbb{N}}(K\cap V_n)=K$. 
\end{proof}

\begin{remark}
Notice that the previous proof also applies to the more general setting of a topology that is weaker than the metric topology in any separable metric space.
\end{remark}

\begin{proposition}
	\label{p:compact-seqcompact}
	Let $(H,d)$ be a Hadamard space that is Hausdorff with respect to $\tau_w$ and let $K\subseteq H$ be bounded. If $K$ weakly compact then $K$ is weakly sequentially compact.
\end{proposition}
\begin{proof}
	Let $(x_n)_{n\in\mathbb N}\subseteq K$, then $(x_n)$ is bounded. By Lemma \ref{l:Bacak} there is a subsequence $(x_{n_k})\overset{w}\to x$ for some $x\in H$. By Theorem \ref{wtopBacak} we have $x_{n_k}\overset{\tau_w}\to x$.
	Assumption $\tau_w$ is Hausdorff  implies that $K$ is weakly closed, consequently $\tau_w$-sequentially closed, hence $x\in K$. 
\end{proof}

As an immediate consequence of Proposition \ref{lemma:seqcomp} and Proposition \ref{p:compact-seqcompact} we obtain:

\begin{corollary}
 \label{c:EberleinSmulyan}
 Let $(H,d)$ be separable and Hausdorff with respect to $\tau_w$. If $K\subseteq H$ is bounded then $K$ is weakly compact if and only if $K$ is weakly sequentially compact.
 \end{corollary}
 
 \begin{corollary}
 	Closed bounded and convex sets in a separable Hadamard space are weakly compact.
 \end{corollary}

\begin{question}
	\label{q:unbounded}
	Is a weakly compact unbounded set $K\subseteq H$ also weakly sequentially compact?
\end{question}
\begin{example}
	This question is motivated by the example of the simplicial tree given by Monod \cite{Monod}. This tree consists of countably many rays of finite but ever increasing length, all meeting at one vertex. If $(V_{i})_{i\in I}$ is some open cover in $\tau_w$ for the simplicial tree then there exists some $i\in I$ such that $V_i$ contains an elementary set that has the common vertex. By construction this elementary set covers most of the space except for a certain ray which can be covered by at most a finite number of elementary sets. Therefore we can always find a finite subcover from our original cover $(V_{i})_{i\in I}$ i.e. the simplicial tree is weakly compact. 
	Moreover this space provides also an example of an unbounded weakly convergent sequence, see discussion in \S 3.1.1 \cite{BerdellimaPhD}.
\end{example}

However we can provide a conditional answer to Question \ref{q:unbounded}.
First we need a regularity assumption.  Let $(H,d)$ be a separable Hadamard space and
 $\{y_n\}_{n\in\mathbb N}\subset H$ a dense set in $H$. We say $(H,d)$ has the nice geodesic structure, if for any $x\in H$ and any sequence $(x_k)_{k\in\mathbb N}$ such that $\lim_kP_{[x,y_n]}x_k=x$ for all $n\in\mathbb N$ implies $\lim_kP_{\gamma}x_k=x$ for any $\gamma\in \Gamma_x(H)$. 

\begin{theorem}
	\label{thm:wcompact-wseqcompact}
	Let $(H,d)$ be a separable weakly proper Hadamard space that has the nice geodesics structure. Then a weakly compact set is weakly sequentially compact. 
\end{theorem}
\begin{proof}
	Let $(H,d)$ be weakly proper and $K\subseteq H$ a weakly compact set. 
	Let $(x_k)_{k\in\mathbb{N}}$ be a sequence in $K$. Then by virtue of \cite[Theorem 3.15]{BerdellimaPhD} $(x_k)_{k\in\mathbb{N}}$ has a weak accumulation point $x\in K$, hence by \cite[Proposition 3.14]{BerdellimaPhD} $x$ is also a weak limit point.
	Since $H$ is separable there exits a dense countable set $\{y_n\}$ in $H$. Let $\gamma_n:[0,1]\to H$ denote the geodesic connecting $x$ with $y_n$ for every $n\in\mathbb{N}$. Now consider the family of sets $V_n$ defined as 
	\begin{equation}
	\label{eq:opensets1}
	V_n:=\bigcap_{i=1}^nU_x(1/n;\gamma_i)\quad\text{where}\quad U_x(1/n;\gamma_i):=\{y\in H\,|\,d(x,P_{\gamma_i}y)<1/n\}.
	\end{equation}
	Since $(H,d)$ is weakly proper then there are $U_{n,i}\in\tau_w$ containing $x$ such that $U_{n,i}\subseteq  U_x(1/n;\gamma_i)$ for every $i=1,2,...,n$ and for every $n\in\mathbb{N}$. Denote by $U_n:=\bigcap_{i=1}^nU_{n,i}$ then $U_n\subseteq V_n$ for all $n\in\mathbb{N}$. 
	Since $U_n$ is a weakly open set containing $x$, it has at least one element $x_{k(n)}$. Passing to a subsequence if necessary we may say that $x_n \in U_n$ for all $n$. In particular $x_n\in V_n$ for all $n$.
	The sets $V_n$ are nested, so we have that $x_m \in V_n$ for all $m \geqslant n$, i.e. $\lim_mP_{\gamma_i}x_m=x$ for all $i=1,2,...,n$ and so $\lim_mP_{\gamma_n}x_m=x$ for all $n\in\mathbb{N}$. Because $(H,d)$ enjoys the nice geodesics structure then $\lim_mP_{\gamma}x_m=x$ for all $\gamma\in\Gamma_x(H)$.	
\end{proof}

As a direc consequence of Theorem \ref{lemma:seqcomp} and Theorem \ref{thm:wcompact-wseqcompact} we obtain: 
\begin{corollary}
	In a separable and weakly proper Hadamard space that enjoys the nice geodesics structure weak compactness and weak sequential compactness coincide.
\end{corollary}
In relation to the above observations we ask the following:
\begin{question}
	Does every separable Hadamard space enjoy the nice geodesics structure?
\end{question}

\subsection{Locally compact spaces}
\label{s:loc-compact}
We now turn to locally compact Hadamard spaces. 
Recall that a topological space $(X,\tau)$ is said to be locally compact if for every $x\in X$ there exists an open set $U\in \tau$ and a compact set $K$ such that $x\in U\subseteq K$. Let $(\mathcal A, \geq)$ be a directed set and $(x_{\alpha})_{\alpha\in\mathcal A}$ a net in $H$ then 
$x_{\alpha}\overset{w}\to x$ if and only if for every $\gamma\in\Gamma_x(H)$ and every $\varepsilon>0$ there is $\alpha_{\varepsilon}\in\mathcal A$ such that $d(x,P_{\gamma}x_{\alpha})<\varepsilon$ for all $\alpha\geq\alpha_{\varepsilon}$ i.e. $x_{\alpha}$ is {\it eventually} in $U_x(\varepsilon;\gamma)$ for every $\gamma\in\Gamma_x(H)$ and every $\varepsilon>0$.

\begin{theorem}\cite[Proposition 4.3]{Kakavandi}\label{theorem:HopfRinow}
The followings statements are equivalent for a Hadamard space $(H,d)$:
\begin{itemize}
\item $(H,d)$ is locally compact.
\item Every closed and bounded subset of $(H,d)$ is compact.
\item Every bounded sequence in $(H,d)$ has a convergent subsequence.
\end{itemize}
\end{theorem}

\begin{lemma}
\label{l:bounded}
 Let $(H,d)$ be locally compact Hadamard space and $(x_{\alpha})_{\alpha\in\mathcal A}$ a net in $H$. If $x_{\alpha}\overset{w}\to x$, then $(x_{\alpha})$ is bounded.
\end{lemma}
\begin{proof}
 Let $x_{\alpha}\overset{w}\to x$, but suppose that $(x_{\alpha})$ is unbounded. Then we can w.l.g. assume that $d(x_{\alpha},x) \geqslant R>0$ for all $\alpha\in \mathcal{A}$. Consider the closed geodesic ball $C:=\mathbb{B}(x,R)$. Denote by $P_{C}x_{\alpha}$ the projection of $x_{\alpha}$ onto $C$ for every $\alpha\in \mathcal{A}$. By assumption $(H,d)$ is locally compact. Theorem \ref{theorem:HopfRinow} implies that $C$ is compact, and in particular the boundary $\partial C$ is compact. Hence, $(P_Cx_{\alpha})$ is a net in the compact set $\partial C$. By \cite[\S 5, Theorem 2]{Kelley} there exits a subnet $(P_Cx_{\beta})_{\beta\in\mathcal B}$ converging to some element $z\in\partial C$. Let $\gamma:[0,1]\to H$ denote the geodesic segment connecting $x$ with $z$. Evidently, $\gamma\subset C$. Denote by $\gamma_{\beta}:[0,1]\to H$ the geodesic segment connecting $x$ with $P_Cx_{\beta}$ for each $\beta\in\mathcal B$. Let $P_{\gamma}x_{\beta}$ denote the projection of $x_{\beta}$ onto the geodesic segment $\gamma$.  
From the triangle inequality we obtain
\begin{equation*}
 d(P_Cx_{\beta},z)\geqslant |d(x_{\beta},P_Cx_{\beta})-d(x_{\beta},z)| ,
\end{equation*}
which implies that $\lim_{\beta}|d(x_{\beta},P_Cx_{\beta})-d(x_{\beta},z)|=0$. 
Since both $C$ and $\gamma$  are strongly closed and convex, we have the following quadratic inequalities (see, e.g., \cite[Theorem 2.1.12]{Bacak}):
\begin{align*}
& d(x_{\beta},P_Cx_{\beta})^2+d(P_Cx_{\beta},P_{\gamma}x_{\beta})^2\leqslant d(x_{\beta},P_{\gamma}x_{\beta})^2 , \\
 \label{ineq2}& d(x_{\beta},P_{\gamma}x_{\beta})^2+d(P_{\gamma}x_{\beta},z)^2\leqslant d(x_{\beta},z)^2 ,
\end{align*}
implying that $d(x_{\beta},P_Cx_{\beta})\leqslant d(x_{\beta},P_{\gamma}x_{\beta})\leqslant d(x_{\beta},z)$. 
Therefore, 
we have that
\begin{equation}\label{ineq:locComp1}
\lim_{\beta}|d(x_{\beta},P_{\gamma}x_{\beta})-d(x_{\beta},P_Cx_{\beta})|=0 .
\end{equation}
By assumption $x_{\alpha}\overset{w}\to x$, and in particular, $x_{\beta}\overset{w}\to x$. Therefore, it follows that $P_{\gamma}x_{\beta}\to x$. Consider the geodesic segment $\eta_{\beta}:[0,1]\to H$ connecting $x$ with $x_{\beta}$. Then there exists $z_{\beta}\in\eta_{\beta}$ such that $z_{\beta}\in\partial C$ for every $\beta\in\mathcal B$. Since $z_{\beta}\in\partial C$, we obtain that $d(x_{\beta},P_Cx_{\beta})\leqslant d(x_{\beta},z_{\beta})$ and thus 
$$d(x_{\beta},x)=d(x_{\beta},z_{\beta})+d(z_{\beta},x)\geqslant d(x_{\beta},P_Cx_{\beta})+R,\quad\forall \beta\in\mathcal B$$
which in turn implies that $|d(x_{\beta},x)-d(x_{\beta},P_Cx_{\beta})|\geqslant R>0$. By the triangle inequality we again have
$$d(P_{\gamma}x_{\beta},x)\geqslant|d(x_{\beta},x)-d(x_{\beta}, P_{\gamma}x_{\beta})|,\quad \forall \beta\in\mathcal B.$$
Therefore,  $\lim_{\beta}P_{\gamma}x_{\beta}=x$ implies that $\lim_{\beta}|d(x_{\beta},x)-d(x_{\beta}, P_{\gamma}x_{\beta})|=0$. Moreover we have 
\begin{align*}
0<R &\leqslant |d(x_{\beta},x)-d(x_{\beta},P_Cx_{\beta})|\\&\leqslant |d(x_{\beta},x)-d(x_{\beta}, P_{\gamma}x_{\beta})|+|d(x_{\beta},P_{\gamma}x_{\beta})-d(x_{\beta},P_Cx_{\beta})| ,
\end{align*}
which together with \eqref{ineq:locComp1} yields a contradiction since the right side vanishes in the limit.
\end{proof}

\begin{corollary}
	\label{c:boundedsequence}
 Let $(H,d)$ be locally compact Hadamard space and $(x_{n})_{n\in\mathbb N}$ a sequence in $H$. If $x_{n}\overset{w}\to x$, then $(x_{n})$ is bounded.
\end{corollary}

\begin{theorem}
\label{wtopstop}
In a locally compact Hadamard space $(H,d)$ weak topology and strong topology coincide.
\end{theorem}
\begin{proof}
Any weakly open set is open. For the converse direction let $U$ be an open set, and suppose that $U$ is not weakly open. Then there is some $x\in U$ such that for any finite collection of geodesic segments $\gamma_1,\gamma_2,\cdots,\gamma_n\in\Gamma_x(H)$ and any $\varepsilon>0$ the set $\bigcap_{i=1}^nU_x(\varepsilon;\gamma_i)$ is not entirely contained in $U$.
let $\mathscr F$ be the collection of all finite subsets from $\Gamma_x(H)$. 
Two sets $F_1,F_2\in\mathscr F$ are said to be equivalent whenever $|F_1|=|F_2|$. Let $[F]$ denote an equivalence class and $|[F]|=n_F$ its cardinality. 
Consider the following sets 
\begin{equation*}
\label{eq:opensets}
V_F=\{y\in C\;:\; d(x,P_{\gamma}y)<1/n_F,\;\gamma\in F\}.
\end{equation*}
By the above observation then for every $F\in\mathscr F$ there exists $x_F\in V_{F}\setminus U$. If we set $F_1\geq F_2$ whenever $F_1\subseteq F_2$ then $(\mathscr F,\geq)$ is a directed set.
In particular $(x_F)_{F\in\mathscr F}$ is a net
and by construction 
$(x_F)_{F\in\mathscr F}$ weakly converges to $x$. From Lemma \ref{l:bounded} the net $(x_F)_{F\in\mathscr F}$ is bounded and consequently by Theorem \ref{theorem:HopfRinow} the (strong) closure $\cl(\{x_{F}\}_{F\in\mathscr F})$ is compact in $H$. Then by virtue of \cite[\S 5, Theorem 2]{Kelley} there is a subnet $(x_{F'})_{F'\in\mathscr F'}$ of the net $(x_F)$ converging to a certain element $y\in \cl(\{x_{F}\}_{F\in\mathscr F})$. In particular $x_{F'}\overset{w}\to y$. By Remark \ref{r:uniqueness}, applied to nets, we have that $y=x$. This is impossible since by construction $x_{F'}\notin U$ for all $F'\in\mathscr F'$.
\end{proof}

\begin{corollary}
	\label{c:weaklyproper}
	Every locally compact Hadamard space is weakly proper.
\end{corollary}
\begin{proof}
	It follows from Theorem \ref{wtopstop} that the elementary sets $U_x(\varepsilon;\gamma)$ are weakly open for every $x\in H, \gamma\in\Gamma_x(H)$ and $\varepsilon>0$. In particular $x$ is an interior point in $\tau_w$ of any elementary set $U_x(\varepsilon;\gamma)$.
\end{proof}

\begin{corollary}
	\label{c:wconv=sconv}
	In a locally compact Hadamard space weak and strong convergence coincide.
\end{corollary}

\subsection{A commentary}
We discuss shortly \cite[Remark 3.28]{BerdellimaPhD} made by the author about another topology which is weaker than the metric topology and somewhat different from $\tau_w$. By Lemma \ref{l:lemma1} the elementary sets $U_x(\varepsilon;\gamma)$ are open in the usual metric topology. If $\mathcal{U}$ is the collection of all finite intersections of sets of the form $U_x(\varepsilon;\gamma)$, where $x$ and $\gamma$ vary, let $\tau$ be the smallest topology generated by $\mathcal{U}$
i.e. any $U\in\tau$ is the union of sets from $\mathcal{U}$. Then any set $U\in \tau$ is open in the usual metric topology. Moreover convergence $\tau$ implies weak convergence (convergence along geodesics) for if $x_n\overset{\tau}\to x$ then for any $\gamma\in\Gamma_x(H)$ and any $\varepsilon>0$ we have $x_n\in U_x(\varepsilon;\gamma)$ for all large enough $n$ i.e. $\lim_{n\to \infty}d(x,P_{\gamma}x_n)=0$ for any $\gamma\in\Gamma_x(H)$; therefore, $x_n\overset{w}\to x$. In particular it holds $\tau_w\subseteq\tau$ and they would coincide if and only if the elementary sets $U_x(\varepsilon;\gamma)$ are open in $\tau_w$. Topology $\tau$ enjoys some properties without any additional assumptions on the space. For example it is Hausdorff and a closed convex set is $\tau$-closed (the proofs for these claims would be identical to the ones for $\tau_w$). Moreover in a locally compact space $\tau=\tau_S$. It is of interest to study also this topology.

\section{The dual space}
\label{s:dual}
\subsection{Construction of the dual and weak-* topology}
The weak topology can also be characterized by considering functions $\phi_{\gamma}(x;\cdot):H\to\mathbb{R}_+$ given by
\begin{equation*}
\phi_{\gamma}(x;y):=d(x,P_{\gamma}y)\; \forall y\in H, \gamma\in \Gamma_{x}(H) .
\end{equation*} 
Indeed, let $\gamma\in\Gamma_{x}(H)$, then  $U_{x}(\varepsilon;\gamma)=\phi_{\gamma}^{-1}(x; \cdot)[0,\varepsilon)$. Notice that nonexpansiveness of projections to convex sets implies that $\phi_{\gamma}$ is Lipschitz continuous with Lipschitz constant equal to $1$. Let 
\begin{equation*}
\|\phi_{\gamma}-\phi_{\eta}\|_{\infty}:=\sup_{y\in H\setminus\{x\}}\frac{|\phi_{\gamma}(x;y)-\phi_{\eta}(x;y)|}{d(x,y)} .
\end{equation*}                                                                                    
Let $H^*_x:=\{\phi_{\gamma}(x;\cdot):\gamma\in\Gamma_x(H)\}$, and let $d_*(\phi_{\gamma}, \phi_{\eta}):= \|\phi_{\gamma}-\phi_{\eta}\|_{\infty}$. 
Notice that since $\phi_{\gamma}(x;y)\leqslant d(x,y)$, the value of $d_*(\phi_{\gamma}, \phi_{\eta})$ is finite.
\begin{lemma} 
	\label{l:dual}
$(\Phi_x(H), d_*$) is a metric space.
\end{lemma}
\begin{proof}
Symmetry and the triangle inequality are evident. In order to show definiteness, suppose that $d_*(\phi_{\gamma},\phi_{\eta})=0$, then this means $\phi_{\gamma}(x;y)=\phi_{\eta}(x;y)$ for all $y\in H$. In particular, we obtain that $\phi_{\eta}(x;\gamma(1))=d(x,\gamma(1))$ and $\phi_{\gamma}(x;\eta(1))=d(x,\eta(1))$. On the other hand, we have the equation $d(x,\gamma(1))= d(x,P_{\gamma}\eta(1))+d(P_{\gamma}\eta(1),\gamma(1))$ and analogously $d(x,\eta(1))= d(x,P_{\eta}\gamma(1))+d(P_{\eta}\gamma(1),\eta(1))$. These two equations imply that $d(P_{\gamma}\eta(1),\gamma(1))=d(P_{\eta}\gamma(1),\eta(1))=0$.  We claim that $\gamma(1)=\eta(1)$. If not then the Alexandrov angle $\angle_{\gamma(1)}([x,\gamma(1)],[\gamma(1),\eta(1)])\geqslant\pi/2$ and analogously $\angle_{\eta(1)}([x,\eta(1)],[\gamma(1),\eta(1)])\geqslant\pi/2$ (see \cite[Theorem 2.1.12]{Bacak} ). On the other hand by the cosine formula for Euclidean triangles, the comparison angles 
$\overline{\angle}_{\gamma(1)}(x,\eta(1))$ and $\overline{\angle}_{\eta(1)}(x,\gamma(1))$ satisfy 
$\overline{\angle}_{\gamma(1)}(x,\eta(1))\geqslant \angle_{\gamma(1)}([x,\gamma(1)],[\gamma(1),\eta(1)])$ and $\overline{\angle}_{\eta(1)}(x,\gamma(1))\geqslant \angle_{\eta(1)}([x,\eta(1)],[\gamma(1),\eta(1)])$. 
Therefore $$\overline{\angle}_{\gamma(1)}(x,\eta(1))\geqslant\pi/2\quad\text{and}\quad\overline{\angle}_{\eta(1)}(x,\gamma(1))\geqslant\pi/2,$$ which is impossible. This completes the proof.
\end{proof}

We call $(H^*_x,d_*)$ the \emph{dual} of $H$ at $x$.
We introduce a topology that is weaker than the one induced by $d_*(\cdot,\cdot)$. The concept of weak-* convergence is defined as follows: A sequence $(\phi_{\gamma_n})\subseteq H^*_x$ is said to weak-* converge to some $\phi_{\gamma}\in H^*_x$ if and only if $\lim_{n\to \infty}\phi_{\gamma_n}(x;y)=\phi_{\gamma}(x;y)$ for all $y\in H$. It is obvious that strong convergence implies weak-* convergence. Weak-* convergence gives rise to a topology which we call the weak-* topology on $H^*_x$. 

\begin{proposition}
	\label{p:weak*}
	A basis for the weak-* topology on $H^*_x$ is determined by the sets 
	\begin{equation}
	\label{eq:openballweak*}
	U_{\gamma}(\varepsilon;y_1,y_2,...,y_n):=\{\phi_{\eta}\in H^*_x: |\phi_{\gamma}(x;y_i)-\phi_{\eta}(x;y_i)|<\varepsilon,\forall i=1,2,..,n\}
	\end{equation}
	where $ y_i\in H,\; \forall i=1,2,...,n;\, n\in\mathbb{N}$
	i.e., any open set in the weak-* topology is a union of sets of the form \eqref{eq:openballweak*}. We denote this topology by $\tau_{w^*}$.
\end{proposition}
\begin{proof}
	Note that $\phi_{\gamma}\in U_{\gamma}(\varepsilon;y_1,y_2,...,y_n)$ for all $\varepsilon>0$ and $y_1,y_2,...,y_n\in H$. Hence $$H^*_x\subseteq \bigcup_{\gamma\in\Gamma_x(H)}U_{\gamma}(\varepsilon;y_1,y_2,...,y_n)\subseteq \bigcup_{\gamma\in\Gamma_x(H),\varepsilon>0, y_1,...,y_n\in H}U_{\gamma}(\varepsilon;y_1,y_2,...,y_n).$$ 
	Therefore the collection of sets $ \{U_{\gamma}(\varepsilon;y_1,y_2,...,y_n): \gamma\in\Gamma_x(H), \varepsilon>0, y_i\in H, i=1,2,...,n\}$ covers $H^*_x$. 
	
	Need to check the intersection property. Let $U_{\gamma_1}(\varepsilon_1;y_1,y_2,...,y_n)$ and $U_{\gamma_2}(\varepsilon_2;z_1,z_2,...,z_m)$ and $\phi_{\eta}\in U_{\gamma_1}(\varepsilon_1;y_1,y_2,...,y_n)\cap U_{\gamma_2}(\varepsilon_2;z_1,z_2,...,z_m)$. There are $0\leqslant s_1,s_2,...,s_n<\varepsilon_1$ and $0\leqslant t_1,t_2,...,t_m<\varepsilon_2$ such that 
	\begin{align*}
	& |\phi_{\gamma_1}(x;y_i)-\phi_{\eta}(x;y_i)|=s_i\\
	& |\phi_{\gamma_2}(x;z_j)-\phi_{\eta}(x;z_j)|=t_j
	\end{align*}
	for $i=1,2,...,n$ and $j=1,2,...,m$. Let $s:=\max\{s_i: i=1,2,...,n\}$ and $t:=\max\{t_j: j=1,2,...,m\}$. Consider the sets $U_{\eta}(\varepsilon_1-s; y_1,...,y_n)$ and $U_{\eta}(\varepsilon_2-t;z_1,z_2,...,z_m)$. Let $\phi_{\sigma}\in U_{\eta}(\varepsilon_1-s; y_1,...,y_n)$ then
	\begin{align*}
	|\phi_{\gamma_1}(x;y_i)-\phi_{\sigma}(x;y_i)|&\leqslant |\phi_{\gamma_1}(x;y_i)-\phi_{\eta}(x;y_i)|+|\phi_{\eta}(x;y_i)-\phi_{\sigma}(x;y_i)|\\&
	< s_i+(\varepsilon_1-s)\leqslant s+(\varepsilon_1-s)=\varepsilon_1, \;\forall i=1,2,...,n
	\end{align*}
	implying $\phi_{\sigma}\in U_{\gamma_1}(\varepsilon_1; y_1,...,y_n)$. Hence $U_{\eta}(\varepsilon_1-s; y_1,...,y_n)\subseteq U_{\gamma_1}(\varepsilon_1; y_1,...,y_n)$. Similarly $U_{\eta}(\varepsilon_2-t; z_1,...,z_m)\subseteq U_{\gamma_2}(\varepsilon_2; z_1,...,z_m)$. Now let $\delta:=\min\{\varepsilon_1-s,\varepsilon_2-t\}$ then 
	\begin{align*}
	\phi_{\eta}\in U_{\eta}(\delta;y_1,...,y_n, z_1,...,z_m)&=U_{\eta}(\delta; y_1,...,y_n)\cap U_{\eta}(\delta; z_1,...,z_m)\\&\subseteq U_{\eta}(\varepsilon_1-s; y_1,...,y_n)\cap U_{\eta}(\varepsilon_2-t; z_1,...,z_m)\\
	&\subseteq U_{\gamma_1}(\varepsilon_1; y_1,...,y_n)\cap U_{\gamma_2}(\varepsilon_2; z_1,...,z_m)
	\end{align*}
	as desired. This completes the proof.
\end{proof}

\begin{theorem} 
\label{wtop*}
The following properties hold:
\begin{enumerate}
 \item A sequence $(\phi_{\gamma_n})$ weak-* converges to $\phi_{\gamma}$ if and only if $\phi_{\gamma_n}\overset{\tau_{w^*}}\to \phi_{\gamma}$.
 \item $(H^*_x, \tau_{w^*})$ is a Hausdorff space.
 \item A weak-* closed set in $H^*_x$ is closed.
\end{enumerate}
\end{theorem}
\begin{proof}
The first property follows from the definition of weak-* topology and that of open sets given in \eqref{eq:openballweak*}.
For the second property it suffices to show that for any two distinct elements $\phi_{\gamma}$ and $\phi_{\eta}$ there is $\varepsilon>0$ such that the open sets $U_{\gamma}(\varepsilon;y)$ and $U_{\eta}(\varepsilon;y)$ have empty intersection for some $y\in H$. Let $\varepsilon:=|\phi_{\gamma}(x;y)-\phi_{\eta}(x;y)|/2$ and suppose there is $\phi_{\mu}\in U_{\gamma}(\varepsilon;y)\cap U_{\eta}(\varepsilon;y)$ then 
  \begin{align*}
   |\phi_{\gamma}(x;y)-\phi_{\eta}(x;y)|&\leqslant |\phi_{\gamma}(x;y)-\phi_{\mu}(x;y)|+|\phi_{\mu}(x;y)-\phi_{\eta}(x;y)|\\&<\frac{1}{2}|\phi_{\gamma}(x;y)-\phi_{\eta}(x;y)|+\frac{1}{2}|\phi_{\gamma}(x;y)-\phi_{\eta}(x;y)|\\&=|\phi_{\gamma}(x;y)-\phi_{\eta}(x;y)|
  \end{align*}
  which is impossible.
In order to show the third property, let $S\subseteq H^*_x$ be a weak-* closed set, and let $\phi_{\gamma_n}$ be a sequence in $S$. Suppose that $\phi_{\gamma_n}\to\phi_{\gamma}$ in the strong topology, then $\phi_{\gamma_n}\overset{w^*}\to\phi_{\gamma}$, which implies that $\phi_{\gamma}\in S$. Therefore $S$ is closed.
\end{proof}
%
%
%
%

\subsection{The case of a Hilbert space} 
\label{ss:Hilbert}
Let $H$ be a Hilbert space $(\mathcal{H}, \|\cdot\|)$ equipped with its canonical norm. Note that for every $x\in \mathcal{H}$ each geodesic segment $\gamma\in\Gamma_x(\mathcal{H})$ corresponds to a unique line $l$ passing through $x$. Given $\gamma,\eta\in\Gamma_x(\mathcal{H})$ we say $\gamma$ is equivalent to $\eta$, and write $\gamma\sim\eta$, if and only if $\gamma,\eta$ belong to the same line $l$. Let $[l]$ denote the equivalence class of all geodesic segments $\gamma\in\Gamma_x(\mathcal{H})$ sharing the same line $l$.
Our aim is to show that our notion of weak convergence along geodesic segments coincides with the usual notion of weak convergence in a Hilbert space. Moreover we prove that a Hilbert space is weakly proper and that the notion of our dual $\mathcal{H}^*_x$ coincides with the usual notion of the dual of a Hilbert space. First we present the following lemma:

\begin{lemma}
\label{l:lines}
 A sequence $(x_n)\subset\mathcal{H}$ weakly converges (in our sense) to $x\in\mathcal{H}$ if and only if $P_{l}x_n\to x$ as $n\uparrow+\infty$ for all lines $l$ containing $x$.
\end{lemma}
\begin{proof}
Suppose that $\lim_nP_lx_n=x$ for all  lines $l$ containing $x$. Let $\gamma\in\Gamma_x(\mathcal{H})$ such that $\gamma\subset l$.  Then all but finitely many of the terms $P_lx_n$ are in the image of $\gamma$, i.e., $P_lx_n=P_{\gamma}x_n$ for all sufficiently large $n$. This means $\lim_nP_{\gamma}x_n=\lim_nP_lx_n=x$. Since this holds for any $l$ containing $x$ and for any $\gamma\in[l]$ then $\lim_nP_{\gamma}x_n=x$ for any geodesic segment $\gamma\in\Gamma_x(\mathcal{H})$. Now let $x_n\overset{w}\to x$. By definition $\lim_nP_{\gamma}x_n=x$ for all $\gamma\in\Gamma_x(\mathcal{H})$. Each $\gamma\in\Gamma_x(\mathcal{H})$ determines a unique line $l$ containing $x$. Then a similar argument shows that $\lim_nP_lx_n=x$ for all lines containing $x$.
\end{proof}
To the collection $\{\phi_{\gamma}(x;\cdot)\}_{\gamma\in[l]}$ we can associate a function $\phi_{[l]}(x;\cdot)$ defined as $ \phi_{[l]}(x;y):=\|x-P_{l}y\|,\hspace{0.2cm}\forall y\in \mathcal{H}$.
Let $\mathcal{H}^*$ denote the dual of a Hilbert space $\mathcal{H}$, that is the space of all bounded linear continuous functionals on $\mathcal H$. A sequence $(x_n)_{n\in\mathbb{N}}\subseteq \mathcal{H}$ is said to converge weakly (in the usual sense) to $x\in\mathcal{H}$ if and only if  $\lim_nf(x_n)=f(x)$ for all $f\in \mathcal{H}^*$.  
\begin{proposition}
	\label{p:Hilbertweakcvg}
	In a Hilbert space weak convergence (in our sense) coincides with the usual notion of weak convergence. 
\end{proposition}
\begin{proof}
Let $x\in \mathcal{H}$ and denote by $\mathcal{H}^*_x:=\{\phi_{\gamma}\;|\;\gamma\in\Gamma_x(\mathcal{H})\}$ where $\phi_{\gamma}(y):=\|x-P_{\gamma}y\|$ for every $y\in \mathcal{H}$.	
By Lemma \ref{l:lines} it follows that $\lim_n\phi_{\gamma}(x;x_n)=0$ for all $\gamma\in[l]$ if and only if $\lim_n\phi_{[l]}(x;x_n)=0$. Therefore it is sufficient to restrict to the family of functions $\{\phi_{[l]}(x;\cdot)\}_L$ where $L$ is the set of all lines containing $x$. Clearly $\{\phi_{[l]}(x;\cdot)\}_L=\mathcal{H}^*_x/\sim$. The projection to lines in Hilbert spaces correspond to inner products with a unit vector along the respective line. Tt follows then that the quotient space $\mathcal{H}^*_x/\sim$ together with its quotient metric coincides with the usual dual $\mathcal{H}^*$ of the Hilbert space $\mathcal{H}$.
\end{proof}

\begin{proposition}
	\label{p:Hilbertweaklyprop}
A Hilbert space is weakly proper. In particular the usual weak topology coincides with $\tau_w$.
\end{proposition}
\begin{proof}
Let $x\in\mathcal{H}$ and consider the elementary set $U_x(\varepsilon;\gamma)$ for some $\varepsilon>0$ and $\gamma\in\Gamma_x(\mathcal{H})$. Let $l$ be the unique line passing through $x$ corresponding to the geodesic segment $\gamma$. Let $y\in U_x(\varepsilon;\gamma)$ and define $l':=l+(y-x)$. Then $l'$ is parallel to $l$. Let $\alpha$ denote the place determined by $l$ and $l'$. Note that for any $z\in \mathcal{H}$ the projections $P_{\alpha}z, P_{l}z$ and $P_{l'}z$ are collinear and moreover their common line $l''$ is perpendicular to both $l$ and $l'$. This argument together with $y\in U_x(\varepsilon;\gamma)$ imply that for some $\delta>0$ and $\gamma'\in\Gamma_y(\mathcal{H})$ belonging to $[l']$ we have $U_y(\delta;\gamma')\subseteq U_x(\varepsilon;\gamma)$. Clearly weak topology coincides with $\tau_w$. 
\end{proof}

\begin{corollary}
	\label{c:productweaklyprop}
	The Cartesian product of a locally compact Hadamard space and a Hilbert space is weakly proper.
\end{corollary}
\begin{proof}
	Follows immediately from Corollary \ref{c:weaklyproper} and Proposition \ref{p:Hilbertweaklyprop}.
\end{proof}
For a given $y\in H$  and $z_l\in l$ let $\theta$ be the the angle between the vectors $y-x$ and $z_l-x$. 
Then from the cosine formula for inner product we get $\langle y-x,z_l-x\rangle=\|y-x\|\|z_l-x\|\cos\theta$. Realizing that $\|y-x\|\cos\theta=\pm\|P_{\gamma}y-x\|$ then follows 
\begin{equation}
 \label{eq:cosine1}
 \pm\phi_{[l]}(x;y)=\frac{1}{\|z_l-x\|}\langle y-x,z_l-x\rangle, \hspace{0.2cm}\forall y\in\mathcal{H}
\end{equation}
Using the linearity of the inner product one can rewrite \eqref{eq:cosine1} as 
\begin{equation}
 \label{eq:cosine2}
 \phi_{[l]}(x;y)=\langle y-x, u_{l}\rangle, \hspace{0.2cm}\forall y\in\mathcal{H}\hspace{0.1cm}\text{where}\hspace{0.1cm} u_{l}:=\pm\frac{z_l-x}{\|z_l-x\|}
\end{equation}
From \eqref{eq:cosine2} we see that $H^*_x/\sim$ coincides with the dual $\mathcal{H}^*$ of the Hilbert space $\mathcal{H}$.

\section{The space of geodesic segments}
\label{s:geo-seg}
\subsection{The metric space of geodesic segments}
For each $x\in H$ the collection of geodesic segments $\Gamma_x(H)$ can be turned into a metric space by equipping it with the metric $d_1(\gamma,\eta):=\sup_{t\in[0,1]}d(\gamma(t),\eta(t))$ where $\gamma,\eta\in\Gamma_x(H)$. 
We call $(\Gamma_x(H),d_1)$ the metric space of geodesic segments and $\Gamma(H):=\bigcup_{x\in H}\Gamma_x(H)$ the bundle of geodesic segments in $H$. 
Let $\psi:\Gamma_x(H)\to H$ be the mapping defined as $\psi(\gamma):=\gamma(1)$ for every $\gamma\in\Gamma_x(H)$. 
\begin{theorem}
	\label{globaliso}
	The mapping $\psi$ is a global isometry from $\Gamma_x(H)$ onto $H$ for every $x\in H$. Therefore, the metric space $(\Gamma_x(H),d_1)$ is a Hadamard space for every $x\in H$.
\end{theorem}
\begin{proof}
	Clearly, $\psi$ is a bijection due to the fact that any two points can be connected by a unique geodesic. It thus suffices to show that $d_1(\gamma,\eta) = d(\gamma(1),\eta(1))$. 
	Consider the comparison triangle $\Delta(\bar{x},\bar{\gamma}(1),\bar{\eta}(1))$. Then the distance $d(\gamma(t), \eta(t))$ is bounded above by the distance between the corresponding points in the comparison triangle, which is in turn is bounded above by the distance between $\bar{\gamma}(1)$ and $\bar{\eta}(1)$. This implies $d(\gamma(t),\eta(t)) \leq d(\gamma(1),\eta(1))$, which proves the claim.  
\end{proof}
Given two geodesics $\gamma_1,\gamma_2\in \Gamma_x(H)$ and $t\in[0,1]$, consider the point $y_t:=(1-t)\gamma_1(1)\oplus t\gamma_2(1)$ and the unique geodesic $\gamma_{t}\in \Gamma_x(H)$ that connects $x$ with $y_t$. We call this geodesic $\gamma_{t}:=(1-t)\gamma_1\oplus t\gamma_2$. Figure \ref{fig:geotriangle} depicts a geodesic segment and a geodesic triangle in $\Gamma_x(H)$.
\begin{figure}
	\centering
	\includegraphics[width=8cm]{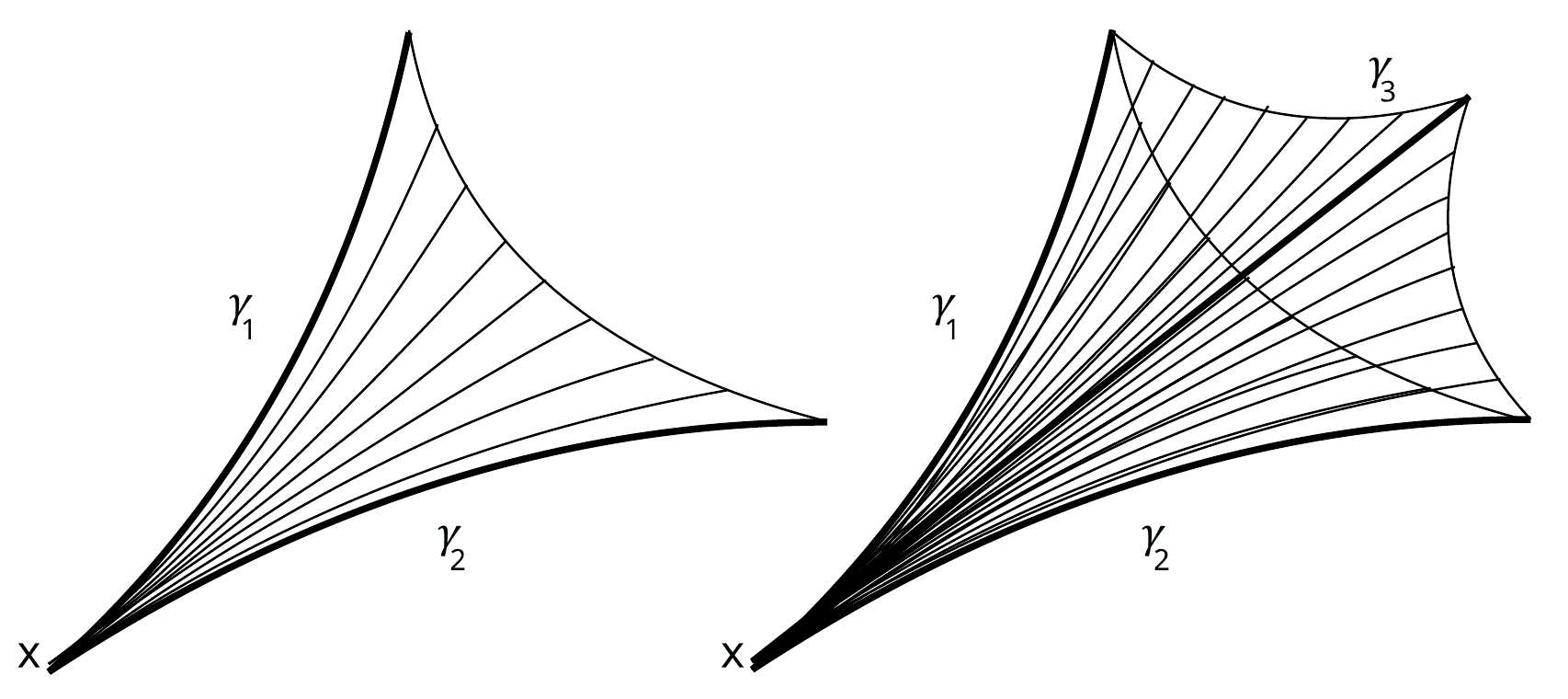} 
	\caption{A geodesic segment in $\Gamma_x(H)$ (left) and a geodesic triangle in $\Gamma_x(H)$ (right).}
	\label{fig:geotriangle}
\end{figure}
\begin{corollary}
	$(\Gamma_x(H),d_1)$ is a complete convex metric space where its convex structure $W$ satisfies $W(\gamma_1,\gamma_2,t)=\gamma_t$ for all $\gamma_1,\gamma_2\in\Gamma_x(H)$ and $t\in[0,1]$.
\end{corollary}

\subsection{Weak convergence in $\Gamma_x(H)$} 
The metric $d_1(\cdot,\cdot)$ induces a strong topology where an open ball of radius $\varepsilon>0$ centered at $\gamma\in\Gamma_x(H)$ is of the form $\mathbb{B}(\gamma,\varepsilon):=\{\eta\in\Gamma_x(H): d_1(\gamma,\eta)<\varepsilon\}$. We can equip the metric space $\Gamma_x(H)$ also with a weak topology in the following way. We say a sequence $(\gamma_n)_{n\in\mathbb{N}}\subseteq \Gamma_x(H)$ weak-$\Gamma$ converges to some element $\gamma\in\Gamma_x(H)$ and we denote it by $\gamma_n\overset{w_{\Gamma}}\to\gamma$ iff $\gamma_n(1)\overset{w}\to \gamma(1)$. We say a set $U\subseteq\Gamma_x(H)$ is weak-$\Gamma$ open if for any $\gamma\in U$ there exists $\varepsilon>0$ and a finite family of geodesics $\gamma_1,...,\gamma_n\in\Gamma_{\gamma(1)}(H)$ such that the set
\begin{equation}
\label{eq:weakgamma}
U_{\gamma}(\varepsilon;\gamma_1,\gamma_2,...,\gamma_n):=\{\eta\in\Gamma_x(H): d(\gamma(1),P_{\gamma_i}\eta(1))<\varepsilon,\forall i\},\hspace{0.2cm}\gamma_i\in\Gamma_{\gamma(1)}(H), n\in\mathbb{N}
\end{equation}
is in $U$. The collection of such sets $U$ indeed defines a topology in $\Gamma_x(H)$. We wish to call it weak-$\Gamma$ topology and denote it by $\tau_{w_{\Gamma}}$. Weak-$\Gamma$ convergence implies convergence in weak-$\Gamma$ topology in the same way as weak convergence implies convergence in the weak topology for the underlying Hadamard space $(H,d)$.
Using the isometry $\psi$ defined in the previous section we can rewrite the definition of weak-$\Gamma$ convergence as $\gamma_n\overset{w_{\Gamma}}\to\gamma$ iff $\psi(\gamma_n)\overset{w}\to\psi(\gamma)$. Similarly $\gamma_n\overset{\tau_{w_{\Gamma}}}\to\gamma$ iff $\psi(\gamma_n)\overset{\tau_w}\to\psi(\gamma)$. The mapping $\psi$ enjoys also the property of being a homeomorphism between the topological spaces $(\Gamma_x(H),\tau_{w_{\Gamma}})$ and $(H,\tau_w)$.
\begin{proposition}
	\label{homeomorphic}
	The topological spaces $(\Gamma_x(H),\tau_{w_{\Gamma}})$ and $(H,\tau_w)$ are homeomorphic.
\end{proposition}
\begin{proof}
	It suffices to prove that the mapping $\psi$ is a bicontinuous function i.e. it is bijective and it has a continuous inverse $\psi^{-1}$. Bijection follows from Proposition \ref{globaliso} since $\psi$ is a global isometry. Now let $\gamma_n\overset{\tau_{w_{\Gamma}}}\to\gamma$ then by definition $\psi(\gamma_n)\overset{\tau_w}\to \psi(\gamma)$ hence $\psi$ is continuous. For the other direction take any sequence $(y_n)_{n\in\mathbb{N}}\subseteq H$ such that $y_n\overset{\tau_w}\to y$ for some $y\in H$. Since $H$ is a uniquely geodesic space for each $y_n$ there is a unique $\gamma_n\in\Gamma_x(H)$ connecting it to $x$. Similarly let $\gamma\in \Gamma_x(H)$ be the geodesic segment connecting $y$ with $x$. Then $y_n\overset{\tau_w}\to y$ means $\gamma_n(1)\overset{\tau_w}\to\gamma(1)$ or equivalently $\psi(\gamma_n)\overset{\tau_w}\to \psi(\gamma)$. By definition of weak-$\Gamma$ topology then $\gamma_n\overset{\tau_{w_{\Gamma}}}\to \gamma$ or equivalently $\psi^{-1}(y_n)\overset{\tau_{w_{\Gamma}}}\to\psi^{-1}(y)$. Hence $\psi^{-1}$ is a continous mapping. 
\end{proof}
Thus $(\Gamma_x(H),\tau_{w_{\Gamma}})$ and $(H,\tau_w)$ are topologically indistinguishable. As a direct result of this we have the following theorem:
\begin{theorem}
	\label{weakgamma}
	The following statements are true:
	\begin{enumerate}[(i)]
		\item $(\Gamma_x(H),\tau_{w_{\Gamma}})$ is Hausdorff whenever $(H,d)$ is weakly proper.
		\item Any bounded sequence has a weak-$\Gamma$ convergent subsequence.
		\item A closed convex set in $\Gamma_x(H)$ is weak-$\Gamma$ sequentially closed. If additionally the set is also bounded then it is weak-$\Gamma$ sequentially compact.
		\item A bounded closed convex set in $\Gamma_x(H)$ is weak-$\Gamma$ compact whenever $H$ is separable.
	\end{enumerate}
	
\end{theorem}
Note that in case $H$ is locally compact then so is the space $\Gamma_x(H)$ and weak-$\Gamma$ topology coincides with the metric topology. In particular any closed and bounded set in $\Gamma_x(H)$ is compact and any bounded sequence $(\gamma_n)_{n\in\mathbb{N}}\subseteq\Gamma_x(H)$ has a convergent subsequence.

\subsection{Existence of the steepest descent} Denote by $e_x$ the trivial geodesic (of zero length) emanating from $x$. Denote by $\Gamma^c_x(H)=\Gamma_x(H)\setminus\{e_x\}$.
A function $f:H\to\mathbb{R}$ is said to be geodesically differentiable at $x\in H$ along geodesic $\gamma\in\Gamma^c_x(H)$ if the following limit exists
\begin{equation}
\label{eq:geodiff}
f'(x;\gamma):=\lim_{y\overset{\gamma}\to x}\frac{f(y)-f(x)}{d(y,x)}.
\end{equation}
In case the limit in \eqref{eq:geodiff} exists for all $\gamma\in \Gamma^c_x(H)$ then we simply say that $f$ is geodesically differentiable at $x$. Note that this limit could vary from one geodesic to another and that $f'(x;\gamma)=f'(x;\eta)$ whenever $\eta\subseteq\gamma$.
Denote by $\overline{\mathbb{B}}(e_x)$ and $\mathbb S(e_x)$ the closed geodesic unit ball and the geodesic unit sphere respectively in $\Gamma_x(H)$ centered at $e_x$. A geodesic segment $\gamma_{\min}\in\mathbb S(e_x)$ of positive length is said to be the direction of steepest descent of the function $f$ at $x\in H$ if $f'(x;\gamma_{\min})=\inf_{\gamma\in\mathbb S(e_x)}f'(x;\gamma)$. Likewise a geodesic segment $\gamma_{\max}\in\mathbb S(e_x)$ is said to be the direction of steepest ascent of the function $f$ at $x\in H$ if $f'(x;\gamma_{\max})=\sup_{\gamma\in\mathbb S(e_x)}f'(x;\gamma)$. 
\begin{theorem}[Consequence of Extreme Value  Theorem]
	\label{descentloccompact}
	Let $(H,d)$ be locally compact and $f:H\to\mathbb{R}$ be a geodesically differentiable function such that $f'(x;\gamma)$ is continuous in $\gamma\in\Gamma^c_x(H)$ for each $x\in H$. Then $f'(x;\gamma)$ is bounded on closed bounded sets in $\Gamma^c_x(H)$ and there are $\gamma_{\max},\gamma_{\min}\in \mathbb S(e_x)$ such that $f'(x;\gamma_{\max})=\sup_{\gamma\in\mathbb S(e_x)}f'(x;\gamma)$ and $f'(x;\gamma_{\min})=\inf_{\gamma\in\mathbb S(e_x)}f'(x;\gamma)$.
\end{theorem}
\begin{proof}
	Let $U\subseteq \Gamma^c_x(H)$ be some arbitrary closed bounded set and assume without loss of generality that $f'(x;\gamma)$ is not bounded from above on $U$. Then there is some sequence $(\gamma_n)_{n\in\mathbb{N}}\subseteq U$ such that $\lim_nf'(x;\gamma_n)=+\infty$. On the other hand $\Gamma_x(H)$ is locally compact since $H$ is. This implies that $U$ is compact because it is closed and bounded thus $(\gamma_n)_{n\in\mathbb{N}}$ has a convergent subsequence $(\gamma_{n_k})_{k\in\mathbb{N}}$. Let $\lim_k\gamma_{n_k}=\gamma$ then $\gamma\in U$. Assumption $f'(x;\gamma)$ is continuous in $\gamma$ implies $\lim_kf'(x;\gamma_{n_k})=f'(x;\gamma)$ but this contradicts $\lim_nf'(x;\gamma_n)=+\infty$. Therefore $f'(x;\gamma)$ must be bounded from above on $U$. Analogue arguments for boundedness from below. Now $\mathbb S(e_x)=\psi^{-1}(\mathbb S(x))$ and by Proposition \ref{homeomorphic} $\psi$ is a homeomorphism so $\mathbb S(e_x)$ is compact. By extreme value theorem, since $f'(x;\gamma)$ is continuous in $\gamma$, there are $\gamma_{\max},\gamma_{\min}\in \mathbb S(e_x)$ such that $f'(x;\gamma_{\max})=\sup_{\gamma\in\mathbb S(e_x)}f'(x;\gamma)$ and $f'(x;\gamma_{\min})=\inf_{\gamma\in\mathbb S(e_x)}f'(x;\gamma)$. This completes the proof.
\end{proof}

\begin{example}
	Let $H=\mathbb R$ and $f(x)=x$ for all $x\in\mathbb R$. Then $\Gamma_x(\mathbb R)=\{\gamma:[0,1]\to\mathbb R\,:\,\gamma(1)\geqslant x\}\cup\{\gamma:[0,1]\to\mathbb R\,:\,\gamma(1)< x\}$, note that $e_x=x$. For every $\gamma\in\Gamma^c_x(\mathbb R)$ it follows that $f'(x;\gamma)=1$ if $\gamma(1)>x$ and $f'(x;\gamma)=-1$ if $\gamma(1)<x$. Let $(\gamma_n)_{n\in\mathbb N}\subseteq \Gamma^c_x(\mathbb R)$ such that $\gamma_n\to\gamma\in \Gamma^c_x(\mathbb R)$, i.e. $\lim_{n\to \infty}d_1(\gamma,\gamma_n)=0$. Suppose w.l.o.g. that $\gamma(1)>x$. This implies $f'(x;\gamma)=1$. By definition of $d_1(\cdot,\cdot)$ we have that $\gamma_n(1)>x$ for all large enough $n$. Hence $f'(x;\gamma_n)=1$ for all large enough $n$. Similarly when $\gamma(1)<x$. Therefore $f'(x;\gamma)$ is continuous in $\gamma\in\Gamma_x(\mathbb R)$ for every $x\in\mathbb R$. Clearly $f$ is geodesically differentiable at every $x\in\mathbb R$. Trivially by Theorem \ref{descentloccompact}, on $\mathbb S(e_x)$ we have that $\gamma_{\max}=[x,x+1]$ and $\gamma_{\min}=[x-1,x]$.
\end{example}	

 By convention we set 
\begin{equation}
\label{eq:trivial-geo}
f'(x;e_x):=\inf_{
	(\gamma_n)_{n\in\mathbb N}\subset\Gamma^c_x(H)}\liminf_{\gamma_n\to e_x}f'(x;\gamma_n).
\end{equation}
First note that if $f'(x;\gamma)$ is lower semicontinuous in $\gamma$ on $\Gamma^c_x(H)$ then it is so on all of $\Gamma_x(H)$ for every $x\in H$. Indeed from \eqref{eq:trivial-geo} we have that $f'(x;e_x)\leqslant \liminf_{\gamma_n\to e_x}f'(x;\gamma_n)$ for any sequence $(\gamma_n)_{n\in\mathbb N}\subset \Gamma^c_x(H)$.

\begin{example} Consider the Hadamard space $(H,d)$ given by $$H:=\{x\in\mathbb{R}^2\,:\,0\leqslant x_1\leqslant 1, 0\leqslant x_2\leqslant 1\}\cup\{x\in\mathbb{R}^2\,:\,-1\leqslant x_1\leqslant 0, -1\leqslant x_2\leqslant 0\}$$ equipped with the length metric $d$.  Let $f:H\to\mathbb{R}$ be given by $f(y)=0$ for $y\in \{x\in\mathbb{R}^2\,:\,-1\leqslant x_1\leqslant 0, -1\leqslant x_2\leqslant 0\}\setminus\{0\}$, $f(y)=c$ for $y\in\{x\in\mathbb R^2\,:\, x_1=0,\;x_2\geqslant 0\}$ where $c$ is some constant, and $f(y):=\|y\|$ otherwise. We demonstrate that $f'(x;\gamma)$ is lower semicontinuous on $\Gamma_x(H)$ for every $x\in H$. It suffices to check on the non-negative quadrant since elsewhere $f$ is identically zero. First let $x\in H$ with $x_1>0$ and $x_2\geqslant 0$. Take $\gamma\in\Gamma_x^c(H)$, then after appropriate simplifications we obtain that 
	\begin{equation}
	\label{eq:derivative}
	f'(x;\gamma)=\frac{\langle\gamma(1)-x,x\rangle}{\|\gamma(1)-x\|\cdot\|x\|}.
	\end{equation}
	Evidently $f'(x;\gamma)$ is continuous in $\gamma\in\Gamma_x^c(H)$ for every such $x\in H$, and by the earlier observation on $f'(x;e_x)$ it follows that $f'(x;\gamma)$ is lower semicontinuous on $\Gamma_x(H)$ for every $x\in H$ with $x_1>0$ and $x_2\geqslant 0$. Lastly we check for $x\in H$ with $x_1=0$ and $x_2\geqslant 0$. If $\gamma\in\Gamma_x(H)$ lies in the ordinate axis then $f'(x;\gamma)=0$ for all $x\neq 0$, i.e. $x_1=0, x_2>0$, since $f$ is identically a constant there. Otherwise formula \eqref{eq:derivative} still holds true. Again $f'(x;\gamma)$ is lower semicontinuous on $\Gamma_x(H)$ for all such $x$. The last case is when $x=0$, i.e. $x_1=x_2=0$, then direct calculations show that $f'(0;\gamma)=0$ for any $\gamma$ contained in the ordinate axis, else $f'(0;\gamma)=1$. All in all we see that $f'(x;\gamma)$ is lower semicontinuous on $\Gamma_x(H)$ for every $x\in H$. 
\end{example}

We say a function $g:\Gamma_x(H)\to\mathbb{R}$ is convex whenever $g(\gamma_{t})\leqslant (1-t)g(\gamma_1)+t g(\gamma_2)$ for any $\gamma_1,\gamma_2\in\Gamma_x(H)$ and $t\in[0,1]$ where $\gamma_{t}$ is the convex combination of $\gamma_1$ and $\gamma_2$. In case $(H,d)$ is not locally compact then a similar result holds if we additionally assume convexity of $f'(x;\gamma)$ in $\gamma$.
\begin{theorem}
	\label{descentgeneral}
	Let $(H,d)$ be a Hadamard space and $f:H\to\mathbb{R}$ be a geodesically differentiable function such that $f'(x;\gamma)$ is convex and lower semicontinuous in $\gamma\in\Gamma_x(H)$ for each $x\in H$. Suppose further that $f'(x;e_x)>-\infty$. Then $f'(x;\gamma)$ is bounded from below on bounded subsets of $\Gamma_x(H)$. Moreover there exists $\gamma_{\min}\in\overline{\mathbb B}(e_x)$ such that $f'(x;\gamma_{\min})=\inf_{\gamma\in\overline{\mathbb B}(e_x)}f'(x;\gamma)$.

\end{theorem}
\begin{proof}
	Let $f$ be geodesically differentiable. Then $f'(x;\gamma)$ exists and it is well defined for all $x\in H$ and $\gamma\in\Gamma^c_x(H)$.  Consider the sub-level set $\lev_{\alpha}:=\{\gamma\in\Gamma_x(H): f'(x;\gamma)\leqslant \alpha\}$ where $\alpha\in\mathbb{R}$. The assumption that $f'(x;\gamma)$ is convex in $\gamma$ implies that $\lev_{\alpha}$ is a convex set. Moreover, since $f'(x;\gamma)$ is lower semicontinuous in $\gamma$, the sub-level sets $\lev_{\alpha}$ are closed.
	By Theorem \ref{weakgamma} (iii) we obtain that $\lev_{\alpha}$ is weak-$\Gamma$ sequentially closed and consequently $f'(x;\gamma)$ is weak-$\Gamma$ sequentially lower semicontinuous in $\gamma$. 
	
	Suppose that $f'(x;\gamma)$ is not bounded from below on bounded subsets of $\Gamma_x(H)$. Then there exists some bounded sequence $(\gamma_n)_{n\in\mathbb{N}}\subseteq \lev_{\alpha}$ such that $f'(x;\gamma_n)<-n$ for all $n\in\mathbb{N}$. By virtue of Theorem \ref{weakgamma} (ii) the sequence $(\gamma_n)_{n\in\mathbb{N}}$ has a weak-$\Gamma$ convergent subsequence $(\gamma_{n_k})_{k\in\mathbb{N}}$. Let $\gamma_{n_k}\overset{w_{\Gamma}}\to\gamma$ then $\gamma\in\lev_{\alpha}$. Weak-$\Gamma$ sequential lower semicontinuity of $f'(x;\gamma)$ together with the assumption $f'(x;e_x)>-\infty$ then implies that $-\infty<f'(x;\gamma)\leqslant\liminf_kf'(x;\gamma_{n_k})\leqslant\liminf_k(-n_k)=-\infty$, which is a contradiction. 
	
	Now let $(\gamma_n)_{n\in\mathbb{N}}\subseteq\overline{\mathbb{B}}^c(e_x)$ be a minimizing sequence. From Theorem \ref{weakgamma} (iii) we have that $\overline{\mathbb{B}}(e_x)$ is weak-$\Gamma$ sequentially compact. Therefore there exists a weak-$\Gamma$ convergent subsequence $(\gamma_{n_k})_{k\in\mathbb N}$. Let $\gamma_{n_k}\overset{w_{\Gamma}}\to\gamma^*\in\overline{\mathbb{B}}(e_x)$ then from the weak-$\Gamma$ sequential lower semicontinuity of $f'(x;\gamma)$ it follows that $\gamma^*=\gamma_{\min}$. 
\end{proof}
From the relation a function $f$ is concave iff $-f$ is convex follows the next corollary which we present without proof.
\begin{corollary}
	If $f'(x;\gamma)$ is concave and upper semicontinuous in $\gamma$ then $f'(x;\gamma)$ is bounded from above on bounded sets and there exists $\gamma_{\max}\in\overline{\mathbb{B}}(e_x)$ such that $f'(x;\gamma_{\max})=\sup_{\gamma\in\overline{\mathbb{B}}(e_x)}f'(x;\gamma)$.
\end{corollary}
\section{Other forms of weak topology}
\label{s:other-topo}
There have been previous attempts to identify weak topologies corresponding to certain notions of weak convergence in Hadamard spaces. These attempts have offered other perspectives on this topic, which we compare to our notion.
 \subsection{Kakavandi's weak topology}
Kakavandi \cite{Kakavandi} proposed a notion of a weak topology, which is based on the following observation. In a Hilbert space $(\mathcal{H},\|\cdot\|)$ equipped with its canonical norm $\|\cdot\|$ a sequence $(x_n)$ converges weakly to an element $x\in \mathcal{H}$ iff $\lim_{n\to \infty}\langle x_n,y\rangle=\langle x,y\rangle$ for all $y\in \mathcal{H}$. This is equivalent to $ \lim_{n\to \infty}\langle x_n-z,y-z\rangle=\langle x-z,y-z\rangle$ for all $y,z\in \mathcal{H}$. Then the identity 
 \begin{equation}
 \label{eq:identity}
 \langle x-z,y-w\rangle=\frac{1}{2}(|x-y|^2+|z-w|^2-|x-w|^2-|z-y|^2)
 \end{equation}
 gives rise to the possibility of extending the definition of weak convergence to a general metric space $(X,d)$ by expressing the right side of \eqref{eq:identity} in terms of the metric $d(\cdot,\cdot)$. Following Berg and Nikolaev \cite{Berg} consider the Cartesian product $X\times X$ where $X$ is a general metric space. Each pair $(x,y)\in X\times X$ determines a so-called \emph{bound vector} which is denoted by $\overrightarrow{xy}$. The point $x$ is called the tail of $\overrightarrow{xy}$ and $y$ is called the head. The zero bound vector is $\overrightarrow{0}_x=\overrightarrow{xx}$. The length of a bound vector $\overrightarrow{xy}$ is defined as the metric distance $d(x,y)$. Furthermore, if $\overrightarrow{u}:=\overrightarrow{xy}$, then $-\overrightarrow{u}:=\overrightarrow{yx}$. Let 
 \begin{equation}
 \label{eq:qlin}
 \langle\overrightarrow{xz},\overrightarrow{yw}\rangle:=\frac{1}{2}(d(x,y)^2+d(z,w)^2-d(x,w)^2-d(z,y)^2) \ . 
 \end{equation}
Kakavandi's notion of  weak convergence is defined in the following way. A sequence $(x_n)$ in a Hadamard space $(H,d)$ converges weakly to an element $x\in H$ if and only if $\lim_{n\to \infty}\langle\overrightarrow{xx_n},\overrightarrow{xy}\rangle=0$ for all $y\in H$. This form of weak convergence coincides with the usual weak convergence in a Hilbert space. 
It turns out, however,  that Kakavandi's notion of convergence does not coincide with $\Delta$-convergence, see Example 4.7 in \cite{Kakavandi}. 

There is a natural topology associated to Kakavandi's convergence generated by sets of the form
\begin{equation}
 \label{eq:Ktopo}
 W(x,y;\varepsilon):=\{z\in H\,|\,|\langle\overrightarrow{xz},\overrightarrow{xy}\rangle|<\varepsilon\},\quad \text{for any}\,x,y\in H,\varepsilon>0.
\end{equation}
More precisely, the family of sets $\{W(x,y;\varepsilon)\,|\, x,y\in H,\varepsilon>0\}$ forms a subbasis for Kakavandi's topology $\tau_K$. One has $x_n\overset{K}\to x$ if and only if $x_n\overset{\tau_K}\to x$, see \cite[Theorem 3.2]{Kakavandi}. Moreover, Kakavandi's topology is Hausdorff.

  \begin{theorem}
  \label{theorem:Kakavandi}
 Let $(H,d)$ be a Hadamard space. Then the followings hold:
 \begin{enumerate}[(i)]
  \item$\tau_w$ is coarser than $\tau_K$.
  \item Kakavandi convergence and weak convergence coincide in a locally compact space.
 \end{enumerate}
 \end{theorem}
 \begin{proof}
 To show (i) let $U\in\tau_w$ be a weakly open set and $x\in U$. By construction of the topology $\tau_w$ there exist a finite number of geodesic segments $\{\gamma_i\}_{i=1}^n$ starting from $x$ such that $\bigcap_{i=1}^nU_x(\delta;\gamma_i)\subset U$. For simplicity suppose $n=1$, i.e., $U_x(\delta;\gamma)\subset U$ for some $\gamma$ starting at $x$. Let $y\in\gamma$ such that $d(x,y)=\delta$ and $\varepsilon:=\delta^2$.
 Consider the open set $W(x,y;\varepsilon)$ in $\tau_K$.  
 Let $z\in W(x,y;\varepsilon)$, then
 $$|\langle\overrightarrow{xz},\overrightarrow{xy}\rangle|= d(x,z)d(x,y)|\cos\theta|<\varepsilon \ , $$
where $\theta$ is the comparison angle at vertex $\overline{x}$ in the Euclidean comparison triangle $ \Delta(\overline{x}, \overline{y}, \overline{z})$ determined by the edge lengths of the triangle $\Delta(x,y,z)$.  Suppose that $d(x,P_{\gamma}z)\geqslant \delta$. Without loss of generality assume that $P_{\gamma}z=y$. Since $\gamma$ is a closed convex set,  by a property of projections we have that the Alexandrov angle $\angle_y([y,x],[y,z])$ at $y$ between the geodesic segments $[y,x]$ and $[y,z]$ is at least $\pi/2$. By nonpositive curvature the comparison angle ${\angle}_{\overline{y}}([\overline{y},\overline{x}],[\overline{y},\overline{z}])$ is greater or equal to $\pi/2$. Then the projection of $\overline{z}$ onto the line $\overline{\gamma}$ that extends beyond the segment $[\overline{x},\overline{y}]$ lies outside this segment, i.e., $|\overline{x}P_{\overline{\gamma}}\overline{z}|\geqslant |\overline{x}\overline{y}|=\delta$. On the other hand, we have 
 $|\overline{x}P_{\overline{\gamma}}\overline{z}|=|\overline{x}\overline{z}|\cos\theta=d(x,z)\cos\theta$, which implies that
 $$|\langle\overrightarrow{xz},\overrightarrow{xy}\rangle|= d(x,z)d(x,y)\cos\theta=|\overline{x}\overline{z}||\overline{x}\overline{y}| \cos\theta=|\overline{x}P_{\overline{\gamma}}\overline{z}||\overline{x}\overline{y}|\geqslant \delta^2=\varepsilon.$$
 This gives a contradiction.
Theorem \ref{wtopstop} and (i) imply (ii).
 \end{proof}
 
 The previous theorem assures that our weak topology $\tau_w$ is coarser than Kakavandi's topology, regardless of whether or not the underlying space is weakly proper. The finer a topology, the fewer compactness results can be expected. Indeed, we are not aware of compactness results for bounded closed convex sets in Kakavandi's topology.

 \subsection{Monod's weak topology} 
Monod \cite{Monod} proposed a topology on a Hadamard space $(H,d)$, which we denote by $\tau_M$. The topology $\tau_M$ is the weakest topology on $(H,d)$ such that any $\tau_s$-closed convex set is $\tau_M$-closed, where $\tau_s$ is the usual metric topology on $(H,d)$. Monod's topology was studied in detail by Kell \cite{Kell} who refers to it as the co-convex topology.  
 The next theorem relates the topologies discussed so far.
 
 \begin{proposition}
 \label{p:threetop}
 The following chain of inclusions holds: $\tau_M\subseteq\tau_w\subseteq\tau_K$.
All three topologies coincide with the usual weak topology whenever $(H,d)$ is a Hilbert space.
 \end{proposition}
 \begin{proof}
 Let  $(H,d)$ be a Hadamard space. Theorem \ref{wtopcvx} implies that a convex set is $\tau_w$-closed if and only if it is $\tau_s$-closed. 
 	Hence $\tau_M\subseteq \tau_w$. The last statement of the theorem follows from the fact that Hilbert spaces are weakly proper together with the fact that $\tau_M$ and $\tau_K$ coincide with the usual weak topology on any Hilbert space, see \cite[Example 18]{Monod}. 
 \end{proof}
It is known that if $K\subseteq H$ is compact then the restrictions of $\tau_M$ and $\tau_s$ to $K$ coincide (see \cite[Lemma 17]{Monod}). Hence, in view of Proposition \ref{p:threetop}, the restrictions of all three weak topologies to a compact subset $K$ of a Hadamard space $H$ coincide with the strong topology. An important property of Monod's topology is that any $\tau_s$-closed convex and bounded set is $\tau_M$-compact, see \cite[Theorem 14]{Monod}. However, it turns out that in general a Hadamard space is not Hausdorff with respect to $\tau_M$. For example, the Euclidean cone of an infinite dimensional Hilbert space is not Hausdorff when equipped with $\tau_M$ \cite[Example 3.6]{Kell}).

 \subsection{Geodesically monotone operators}
A continuous operator $T:H\to H$ is said to be geodesically monotone if for all $x_0,x_1\in H$ the real-valued function $\varphi:[0,1]\to\mathbb{R}_+$ defined by $\varphi(\alpha;x_0, x_1):=d(Tx_0,Tx_{\alpha})$ is monotone in $\alpha$ where $x_{\alpha}:=(1-\alpha) x_0\oplus \alpha x_1$ is the convex combination along the geodesic from $x_0$ to $x_1$.
The next theorem provides a sufficient condition for an arbitrary Hadamard space to be Hausdorff in Monod's topology.
 \begin{theorem} 
 \label{geomonotone}
 If the projection $P_{\gamma}$ is geodesically monotone for all geodesic segments $\gamma$, then $(H,\tau_M)$ is Hausdorff. 
 \end{theorem}
 \begin{proof}
 Let $x,y\in H$ be two distinct points and $\gamma:[0,1]\to H$ a geodesic such that $\gamma(0)=x,\gamma(1)=y$. Let $l>0$ denote the length of $\gamma$. For some fixed number $0<\varepsilon<l$ let 
 $C(x,\varepsilon):=\{z\in H\,|\, d(x,P_{\gamma}z)\leqslant \varepsilon\}$. We claim that $C(x,\varepsilon)$ is a closed convex set. Closedness follows immediately since $P_{\gamma}$ is nonexpansive and therefore continuous.  Let $z_0,z_1\in C(x,\varepsilon)$ be two distinct elements. Let $z_{\alpha}:=(1-\alpha) z_0\oplus \alpha z_1$ for some $\alpha\in(0,1)$. By assumption, $P_{\gamma}$ is a geodesically monotone operator; thus, $d(P_{\gamma}z_{0},P_{\gamma}z_{\alpha})$ is monotone in $\alpha$, implying  that $P_{\gamma}z_{\alpha}\in[P_{\gamma}z_0,P_{\gamma}z_1]$. The estimate $d(x,P_{\gamma}z_{\alpha})\leqslant\max\{d(x,P_{\gamma}z_0), d(x,P_{\gamma}z_1)\}\leqslant\varepsilon$ implies that $z_{\alpha}\in C(x,\varepsilon)$. By definition of $\tau_M$ it follows that $C(x,\varepsilon)$ is $\tau_M$-closed . Hence $H\setminus C(x,\varepsilon)$ is $\tau_M$-open. By construction $y\in H\setminus C(x,\varepsilon)$. Using the same argument, it follows that $C(y,\varepsilon):=\{z\in H\;|\;d(y,P_{\gamma}z)\leqslant l-\varepsilon\}$ is $\tau_M$-closed. Hence, $H\setminus C(y,\varepsilon)$ is a $\tau_M$-open set containing $x$. It is evident by construction that $(H\setminus C(x,\varepsilon))\cap (H\setminus C(y,\varepsilon))=\emptyset$. 
Therefore, $(H,\tau_M)$ is a Hausdorff space. 
 \end{proof}
 The contrapositive of this statement 
 together with \cite[Example 3.6]{Kell} shows that the projection $P_{\gamma}$ is not a geodesically monotone operator in a general Hadamard space. This is in contrast with projections to geodesic segments  in Hilbert spaces, which are always geodesically monotone. Notice that the converse of Theorem \ref{geomonotone} is not true as Figure \ref{fig:counterexample} shows.  Another interesting implication from the Example illustrated in Figure \ref{fig:counterexample} relates to the so called {\em normal cone}. Given a closed convex set $C\subseteq H$ and $p\in C$ we define the normal cone at $p\in C$ and denote it by $N(p,C)$ as the set of all elements in $H$ such that the geodesic segment connecting the given element with the point $p$ makes an Aleksandrov angle no less then $\pi/2$ with any geodesic segment connecting $p$ with another point in the set $C$, i.e.
 \begin{equation}
 \label{eq:normalcone}
 N(p,C):=\{x\in H\,:\, \angle_p([x,p],[p,y])\geqslant\pi/2,\,\forall y\in C\}.
 \end{equation}
 Figure \ref{fig:counterexample} tells us that $N(p,C)$ is not convex. This is in contrast with a basic result in Hilbert spaces that $N(p,C)$ is always a closed convex set. 
 \begin{figure}[h!]
\centering
\includegraphics[width=8cm]{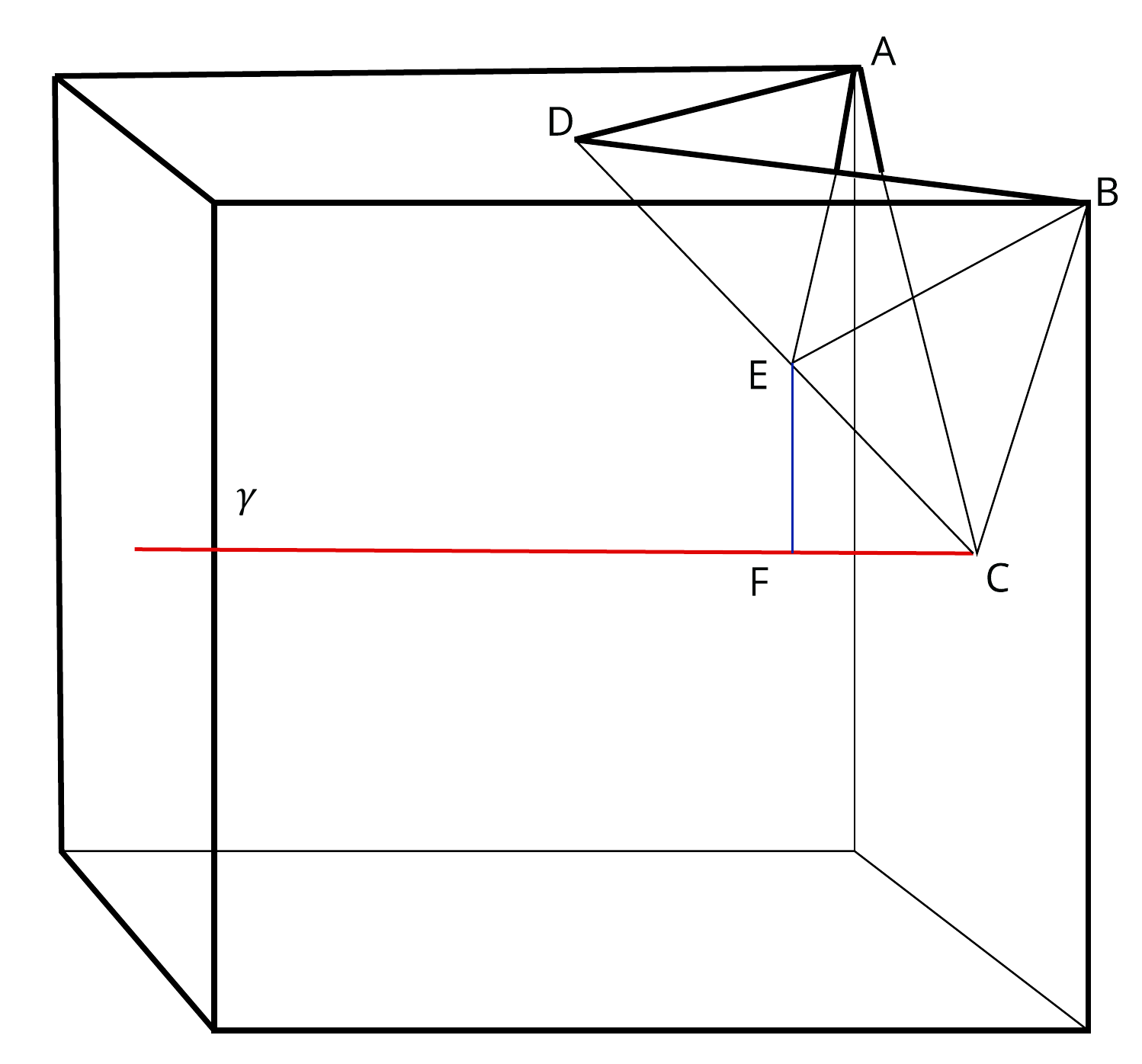}
\caption{An example of a Hadamard space that arises by removing the solid wedge $(ABCD)$ from a solid cube. Both endpoints $A$ and $B$ of the geodesic segment $[A,B]$ project to the point $C$ on the geodesic $\gamma$; however, the midpoint $E$ of $[A,B]$ projects to the point $F$, which is distinct from $C$. This illustrates that projection to geodesics might not preserve monotonicity.}\label{fig:counterexample}
\end{figure}

\begin{remark}
	\label{r:nice-property} Geodesic monotonicity of metric projections onto geodesics coincides with the so-called ''(nice) N-property'' introduced in \cite{Espi}.
\end{remark}

\subsection{Discussion} 
\label{ss:discussion}
In relation to weak convergence of bounded sequences in $\CAT(0)$ spaces \cite[Lytchak--Petrunin]{Lytchak} recently introduced a topology $\tau$ on a $\CAT(0)$ space $(X,d)$ in the following way: a set $S\subseteq X$ is $\tau$-closed in $X$ if any bounded sequence $(x_n)_{n\in\mathbb N}$ in $S$ that weakly converges to $x\in X$ implies $x\in S$. By construction $\tau$ is finer than $\tau_w$. This topology differs from $\tau_w$ in that $\tau$ characterizes only bounded sequences (nets). Now concerning the example in \cite[\S 4]{Lytchak} it is proved that it is not Hausdorff in $\tau$, consequently it cannot be Hausdorff in $\tau_w$ since the latter is a coarser topology. By Lemma \ref{wtopHaus} it follows then that this example of Hadamard space cannot be weakly proper. Therefore not all Hadamard spaces are weakly proper.

\newpage
 \bibliographystyle{amsplain}
  \bibliography{literature}



\bigskip
\subsection*{Acknowledgements} I am very grateful to my advisers for their unconditional help during the writing of my PhD. I thank also the anonymous referee of my PhD thesis for their corrections. Credit should be given to Ph. Miller who noted that the assumption of weakly properness in Theorem \ref{wtopcvx} would have been superfluous.

\subsection*{Funding} This research was supported by DAAD (Deutscher Akademischer Austauschdienst) scholarship and partially by DFG under Germany's Excellence Strategy--The Berlin Mathematics Research Center MATH+ (EXC-2046/1,  Projektnummer:  390685689).
\bigskip

\end{document}